\documentclass[12pt]{amsart}
\usepackage{amssymb, amscd, amsmath, amsthm, 
epsfig, latexsym, enumerate}

\renewcommand{\geq}{\geqslant}
\renewcommand{\leq}{\leqslant}

\newtheorem{theorem}{Theorem}
\newtheorem{lemma}[theorem]{Lemma}
\newtheorem{cor}[theorem]{Corollary}

\newtheorem*{cor*}{Corollary}

\begin{document}

\title{$PD_3$-pairs with compressible boundary}

\author{Jonathan A. Hillman }
\address{School of Mathematics and Statistics\\
     University of Sydney, NSW 2006\\
      Australia }

\email{jonathan.hillman@sydney.edu.au}

\begin{abstract}
We extend work of Turaev and Bleile to relax the $\pi_1$-injectivity hypothesis
in the characterization of the fundamental triples of 
$PD_3$-pairs with all boundary components aspherical. 
This is further extended to pairs $(P,\partial{P})$ which also have 
spherical boundary components and with $c.d.\pi_1(P)\leq2$.
\end{abstract}

\keywords{Algebraic Loop Theorem, free factor, $PD_3$-pair, peripheral system}

\subjclass{57P10}

\maketitle

A {\it $PD_n $-complex} is a space homotopy equivalent to a cell complex 
which satisfies Poincar\'e duality of formal dimension $n$ with local coefficients. 
Such spaces model the homotopy types of $n$-manifolds.
The lowest dimension in which there are $PD_n$-complexes which 
are not homotopy equivalent to closed $n$-manifolds is $n=3$.
The homotopy type of a $PD_3$-complex $P$ is determined by 
$\pi=\pi_1(P)$,
the orientation character $w=w_1(P)$ and the image $\mu$ 
of the fundamental class in $H_3(\pi;\mathbb{Z}^w)$ \cite{He}.
Turaev formulated and proved a Realization Theorem, 
characterizing the triples $[\pi,w,\mu]$ which arise in this way,
in terms of homotopy equivalences of modules.
He also gave a new proof of Hendriks' Classification Theorem, 
and used these results to establish Splitting and 
Unique Factorization Theorems parallel to those known for 3-manifolds 
\cite{Tu}.

$PD_3$-pairs model the homotopy types of 3-manifolds with boundary.
The basic invariant is again a fundamental triple $[(\pi,\{\kappa_j\}),w,\mu]$,
where the fundamental group data now includes
the homomorphisms induced by the inclusions of the boundary components.
These triples satisfy an extension of the Turaev condition,
and also a boundary compatibility condition, 
that $w$ restricts to the orientation characters of the boundary components,
and the image of $\mu$ under the connecting homomorphism 
(for an appropriate long exact sequence of homology) 
is a fundamental class for the boundary.
Bleile extended the Classification and Realization Theorems 
to $PD_3$-pairs with aspherical boundary components.
Her version of the Realization Theorem requires that the homomorphisms 
$\kappa_j$ be injective.
(For a 3-manifold this corresponds to having incompressible boundary.)
She also gave two Decomposition Theorems, 
corresponding to interior and boundary connected sums, respectively \cite{Bl}.

We shall show that in the orientable case the $\pi_1$-injectivity 
restriction may be replaced by necessary conditions imposed by 
the Algebraic Loop Theorem of Crisp,
which asserts that the kernel of the map on fundamental groups induced by the inclusion of a boundary component $Y$ of a $PD_3$-pair is 
the normal closure of finitely many disjoint essential simple closed curves 
on $Y$ \cite{Cr07}.
Our Realization Theorem (Theorem \ref{asph boundary} below) asserts that  
a fundamental triple is realizable by a $PD_3$-pair with aspherical boundary 
if and only if it satisfies the Turaev and boundary compatibility conditions 
and the conclusion of the Algebraic Loop Theorem,
provided that $\pi$ has a large enough free factor
and has no orientation-reversing element of order 2.
(The condition on free factors is known to be necessary for
3-manifold pairs.)
The key ideas are a criterion for recognizing free factors and the inductive application of the arguments of the Decomposition Theorems of \cite{Bl}.

We expect that orientable pairs $(P,\partial{P})$ with aspherical boundary
should be  interior
connected sums of $PD_3$-complexes and pairs $(P_i,\partial{P}_i)$ 
with $P_i$ aspherical.
This expectation holds after stabilization by connected sums with copies of $S^2\times{S^1}$,
and would hold unreservedly if we could establish an inequality 
suggested by Lemma \ref{min gen} below.
(The need for stabilization arises from the difference between 
the algebraic and topological  Loop Theorems, as presently understood. 
See \S6 below.)
Much of the argument applies also to non-orientable pairs,
at least if there is no orientation-reversing element of order 2 in $\pi_1(P)$.

The Realization Theorem extends immediately to allow some 
$S^2$ boundary components, 
since capping off spheres does not change the fundamental group, 
and the fundamental class extends uniquely to the resulting pair 
with aspherical boundary components.
The fundamental triple remains a complete invariant when $\pi$ has cohomological dimension $\leq2$.
(This is so when $\pi$ is torsion free and the pair has no summand 
which is an aspherical $PD_3$-complex.)
In particular, $PD_3$-pairs with free fundamental group 
are homotopy equivalent to 3-manifolds with boundary.
However in the remaining cases we appear to need also a $k$-invariant.
Beyond this, there remains the issue of classification and realization 
of $PD_3$-pairs with $RP^2$ boundary components. 
This seems just out of reach for the moment.

We present the main definitions and summarize the results of Bleile, 
Crisp and Turaev in \S1.
The next three sections contain preparatory steps towards the main result.
In \S2 we define the end module $R(\pi)$ and in \S3 we estimate 
the rank of the maximal free factor of $\pi$ in terms of the peripheral data.
This is applied  in conjunction with a criterion of Dunwoody \cite[Corollary IV.5.3]{DD}.
The next section presents two notions of connected sum 
and the related notion of adding a handle, 
and states the results of Bleile that we shall use inductively in this and 
the next section.
We also prove a key lemma.
This is the one point in this paper in which the Turaev condition is invoked explicitly,
and our argument is modelled in part on that of \cite[\S5]{Bl}.
The next section \S5 contains the main result (Theorem \ref{asph boundary}),
and the following short section has several examples indicating the independence of the hypotheses.
In \S7 we show that there are only a handful of non-orientable
$PD_3$-pairs with $\pi$ indecomposable and virtually free.
The remaining sections consider pairs whose boundary is not aspherical.
In \S8 we give a Classification Theorem for pairs with no $RP^2$
boundary components,
and in \S9 we sketch a difficulty in extending this result further.
In \S10 we settle the cases with $\pi$ finite
and in the final section we consider briefly the role 
of $\pi$ alone in determining the pair.

\section{necessary conditions}

Let $(P,\partial{P})$ be a $PD_3$-pair, 
and let $w=w_1(P):\pi=\pi_1(P)\to\mathbb{Z}^\times$ be the orientation character.
Then $H_3(P;\partial{P};\mathbb{Z}^w)\cong\mathbb{Z}$,
and cap product with a fundamental class $[P,\partial{P}]\in
{H_3(P;\partial{P};\mathbb{Z}^w)}$ induces isomorphisms
from $H^i(P;\overline{\mathcal{R}})$ to $H_{3-i}(P,\partial{P};{\mathcal{R}})$
for any right $\mathbb{Z}[\pi]$-module $\mathcal{R}$.
(Here $\overline{\mathcal{R}}$ is the left module obtained from 
$\mathcal{R}$ via the $w$-twisted involution,
with $\pi$-action given by $g.r=w(g)r.g^{-1}$ for $r\in\mathcal{R}$ 
and $g\in\pi$.)
We shall reserve the term ``$PD_3$-complex" for $PD_3$-pairs 
with empty boundary.
It is a standard consequence of Poincar\'e duality and the long exact sequence
of homology for the pair that $\chi(\partial{P})=2\chi(P)$.

We shall assume always that the  {\it ambient\/} space $P$ is connected, 
but in general the {\it boundary\/} $\partial{P}$ may have several components.
Each component of $\partial{P}$ is a $PD_2$-complex with 
orientation character the restriction of $w$, and the choice of
a fundamental class $[P,\partial{P}]$ determines fundamental classes 
for the boundary components whose sum is the image of $[P,\partial{P}]$ 
in $H_2(\partial{P};\mathbb{Z}^w)$. 
We may assume that each component of $\partial{P}$ 
is a closed, connected 2-manifold with a collar neighbourhood.
(This can always be arranged, by a mapping cylinder construction,
since $PD_2$-complexes are homotopy equivalent to closed surfaces.)
Let $\partial\widetilde{P}$ be the preimage of $\partial{P}$
 in the universal cover $\widetilde{P}$.
 
Let $\partial{P}=\amalg_{j\in{J}}{Y_j}$ and let
$\kappa_j:\pi_1(Y_j)\to\pi$ be the homomorphism induced by inclusion, 
and let $B_j=\mathrm{Im}(\kappa_j)$ be the corresponding {\it peripheral subgroup},
for all $j\in{J}$.
(We include the trivial homomorphisms corresponding to $S^2$ boundary 
components, to record such components.
However, in \S3--\S7 we shall assume that all the boundary components are aspherical.)
Since we must choose paths connecting  each boundary component 
to the base-point for $P$,
the homomorphisms $\kappa_j$ are only well-defined up to conjugacy.
We shall assume that a fixed choice is made, when necessary.
The {\it peripheral system} $\{\kappa_j|j\in{J}\}$ is $\pi_1$-{\it injective}
if $\kappa_j$ is injective for all $j\in{J}$,
while if the subgroups $B_j$ are all torsion free then $(P,\partial{P})$
is {\it peripherally torsion free}.

It is easy to see that if $P$ is aspherical and $\pi=1$ then 
$(P,\partial{P})\simeq(D^3,S^2)$. 
The pair $(P,\partial{P})$ is {\it aspherical\/} if $P$ is aspherical and 
$\pi\not=1$. 
It has {\it aspherical boundary\/} if every component 
of $\partial{P}$ is aspherical.
Aspherical $PD_3$-pairs have aspherical boundary \cite[Lemma 3.1]{Hi}.

We shall say that the orientation character $w$ {\it splits\/} if $w(g)=-1$ 
for some $g\in\pi$ such that $g^2=1$.
Pairs for which $w$ does not split are peripherally torsion free,
but the converse is false.
For instance, $P=RP^2\times{S^1}$ has no boundary, 
and so is peripherally torsion free,
but the inclusion of $RP^2\times\{*\}$ splits $w_1(P)$.
Let $(P^+,\partial{P}^+)$ be the orientable covering pair
with fundamental group $\pi^+=\mathrm{Ker}(w)$.

The {\it fundamental triple\/} of the pair is $[(\pi,\{\kappa_j\}),w,\mu]$,
where $\mu$ is the image of $[P,\partial{P}]$ 
in $H_3(\pi,\{\kappa_j\};\mathbb{Z}^w)$.
The {\it Classification Theorem\/} asserts that two $PD_3$-pairs with aspherical boundary are homotopy equivalent if and only if 
their fundamental triples are isomorphic \cite[Theorem 1.1]{Bl}.

There are three types of condition which are necessary for a triple 
to be realized by a $PD_3$-pair: 
those imposed by the Algebraic Loop Theorem,
the boundary compatibilities of fundamental classes, 
and the Turaev condition on the fundamental class.
We shall outline these conditions (and describe the relative homology group
$H_3(\pi,\{\kappa_j\};\mathbb{Z}^w)$) in the following paragraphs. 

The {\it Algebraic Loop Theorem\/} asserts that if $(P,\partial{P})$ is a $PD_3$-pair and $Y$ is an aspherical boundary component then 
there is a finite maximal family $Ess(Y)$ of free homotopy classes of 
disjoint essential simple closed curves on $Y$ which are 
each null-homotopic in $P$ \cite{Cr07}.
If $\gamma$ is a simple closed curve on $Y$ which is null-homotopic in $P$ 
then $w(\gamma)=1$, and so $\gamma$ is orientation-preserving.
If it is non-separating, 
then there is an associated separating curve in $Ess(Y)$,
bounding a torus or Klein bottle summand. 
The subgroup $\mathrm{Im}(\kappa)=\kappa(\pi_1(Y))$ is a 
non-trivial free product of $PD_2$-groups,
copies of $\mathbb{Z}/2\mathbb{Z}$ and free groups,
by the Algebraic Loop Theorem and the argument of 
\cite[Proposition 1.1]{JS76}.
(See also \cite[Lemma 3.1 and Corollary 3.10.2]{Hi}.)
Each indecomposable factor of $\mathrm{Im}(\kappa)$
which is not infinite cyclic is conjugate in $\pi$ 
to a subgroup of one of the factors of $\pi$, 
by the Kurosh Subgroup Theorem.
(If each curve in $Ess(Y)$ separates $Y$ then $\mathrm{Im}(\kappa)$ 
has no free factors.)

It is a familiar consequence of the topological Loop and Sphere Theorems 
that every compact orientable 3-manifold with boundary is 
the interior connected sum of indecomposable 3-manifolds, 
and these in turn either have empty boundary or may be assembled from 
aspherical 3-manifolds with $\pi_1$-injective boundary by adding 1-handles.
However, in our setting we do not know whether there is a Poincar\'e embedding of a family of 2-discs representing the classes in $Ess(Y)$,
along which we could reduce the pair to one with $\pi_1$-injective boundary.
The  Decomposition Theorems of Bleile allow for connected sum decompositions,
provided that the potential summands are indeed $PD_3$-pairs \cite{Bl}.
(When the boundary is aspherical and $\pi_1$-injective this 
is ensured by the Realization Theorem.)

Let $Y$ be a closed surface.
A homomorphism $\kappa:S=\pi_1(Y)\to{G}$ is {\it geometric\/} 
if there is a finite family $\Phi$ of disjoint, 
2-sided simple closed curves on $Y$ such that $\mathrm{Ker}(\kappa)$ 
is normally generated by the image of $\Phi$ in $S$.
It is {\it torsion free geometric\/} if $\kappa(S)$ is torsion free.
A {\it geometric group system\/} is a pair $(G,\mathcal{K})$, 
where $G$ is a finitely presentable group and $\mathcal{K}$ is 
a finite set of geometric homomorphisms from $PD_2$-groups to $G$.
It is {\it trivial\/} if $G=1$.
Two geometric group systems $(G_1,\mathcal{K}_1)$ and 
$(G_2,\mathcal{K}_2)$  are {\it isomorphic\/} if there is an isomorphism 
$\theta:G_1\to{G_2}$ and a bijection $b:\mathcal{K}_1\to\mathcal{K}_2$
such that $\theta\circ\kappa$ and $b(\kappa)$ have the same domains
and are conjugate as homomorphisms into $G_2$, 
for all $\kappa\in\mathcal{K}_1$.
A homomorphism $w:G\to\mathbb{Z}^\times$ is an {\it orientation character\/} for $(G,\mathcal{K})$ if $w\circ\kappa=w_1(Y)$ for each 
$\kappa\in\mathcal{K}$.

The peripheral system of a $PD_3$-pair with aspherical boundary
is a geometric group system, 
by the Algebraic Loop Theorem.
It is torsion free geometric if the pair is peripherally torsion free.

Let $(G,\{\kappa_j|j\in{J}\})$ be a geometric group system,
and let $f_j:Y_j=K(S_j,1)\to{K(G,1)}$ be maps realizing the homomorphisms 
$\kappa_j$.
Let $K$ be the mapping cylinder of $\amalg{f_j}:Y=\amalg{Y_j}\to {K(G,1)}$,
and let 
$H_*(G,\{\kappa_j\};M)=H_*(M\otimes_{\mathbb{Z}[G]}C_*(K,Y;\mathbb{Z}[G]))$,
for any right $\mathbb{Z}[G]$-module $M$.
If $G$ and the $S_j$s are $FP_2$ then $C_*(K,Y;\mathbb{Z}[G])$
is chain homotopy equivalent to a free complex which is finitely generated 
in degrees $\leq2$.
Let $w$ be an orientation character for $(G,\mathcal{K})$ 
and let $\mu\in{H_3(G,\{\kappa_j\};\mathbb{Z}^w)}$.
Then  $\mu$ {\it satisfies the boundary compatibility condition\/} 
if its image in $H_2(Y;\mathbb{Z}^w)$ under the connecting homomorphism
for the exact sequence of the pair $(K,Y)$ is a fundamental class for $Y$.

Let $I(G)$ be the kernel of the augmentation homomorphism 
from $\mathbb{Z}[G]$ to $\mathbb{Z}$,
and let  $C_*$ be a free left $\mathbb{Z}[G]$-chain complex which is finitely generated in degrees $\leq2$.
Let $C^*$ be the dual cochain complex,
defined by $C^q=Hom(C_q,\mathbb{Z}[G])$, for all $q$.
Let $F^2(C_*)=C^2/\delta^1(C^1)$.
Then Turaev defined a homomorphism 
\[
\nu_{C_*,2}:H_3(\mathbb{Z}^w\otimes_{\mathbb{Z}[G]}C_*)
\to[F^2(C_*),I(G)],
\]
where $[A,B]$ is the abelian group of projective homotopy equivalence 
classes of $\mathbb{Z}[G]$-module homomorphisms.
A class $\mu\in{H_3(G,\{\kappa_j\};\mathbb{Z}^w)}$ 
{\it satisfies the Turaev condition\/} if $\nu_{C_*,2}(\mu)$ is 
a projective homotopy equivalence for some free chain complex 
$C_*\simeq{C_*(K,Y;\mathbb{Z}[G])}$
which is finitely generated in degrees $\leq2$.
This condition does not depend on the choice of $C_*$
within the chain homotopy type.

The {\it Realization Theorem\/} of Turaev asserts that a triple $[G,w,\mu]$ 
is the fundamental triple of a $PD_3$-complex if and only if $\mu$
satisfies the Turaev condition.
In this case $\nu_{C_*,2}$ is an isomorphism,
and so a group $G$ is the fundamental group of a $PD_3$-complex
if and only if certain modules are stably isomorphic \cite{Tu}.
(See also \cite[Theorem 2.4]{Hi}.)

The fundamental triple of a $PD_3$-pair satisfies 
the boundary compatibility and Turaev conditions,
and Bleile's extension of the Realization Theorem asserts that triples with 
$\pi_1$-injective peripheral system and which satisfy these conditions
are realizable by $PD_3$-pairs \cite{Bl}.
The Algebraic Loop Theorem says nothing of interest in this case.
However,  it is the key to a much wider result.
(In the $\pi_1$-injective case $\nu_{C_*,2}$ is again an isomorphism.
In general, this is not so.
As a consequence, 
we do not yet have a realization theorem for the 
peripheral system alone.)

Let $(P,\partial{P})\sharp(Q,\partial{Q})$ 
and $(P,\partial{P})\natural(Q,\partial{Q})$ 
be the interior and boundary connected sums of pairs $(P,\partial{P})$ and
$(Q,\partial{Q})$, respectively. (See \cite{Bl,Wa}.)
All Betti numbers shall refer to homology with coefficients $\mathbb{F}_2$,
and shall be written as $\beta_i(X)$, rather than $\beta_i(X;\mathbb{F}_2)$,
for convenience in including the non-orientable case,
and for simplicity of notation.
We shall abbreviate ``cohomological dimension of $\pi$" as ``$c.d.\pi$".

Much of this material is in Chapters 2 and 3 of the book \cite{Hi}.

\section{the end module}

Let $(P,\partial{P})$ be a $PD_3$-pair, and let $\pi=\pi_1(P)$.
We begin with a very simple lemma.

\begin{lemma}
\label{null boundary}
Let $(P,\partial{P})$ be a $PD_3$-pair with a 
null-homotopic boundary component.
Then $(P,\partial{P})\simeq(D^3,S^2)$.
Hence an aspherical $PD_3$-pair is not a proper interior connected sum.
\end{lemma}

\begin{proof}
On considering the long exact sequences of homology for the pairs 
$(P,\partial{P})$ and $(\widetilde{P},\partial\widetilde{P})$ we see that
$\partial{P}$ is connected and then that $\pi_1(P)=1$. 
Hence $H_2(P,\partial{P};\mathbb{Z})=0$, by Poincar\'e duality,
and so $H_2(P;\mathbb{Z})=0$, 
by the long exact sequence of the pair and the hypothesis on $\partial{P}$.
Therefore $P$ is contractible and so $(P,\partial{P})\simeq(D^3,S^2)$.

If $(Q,\partial{Q})$ is a proper interior connected sum then 
the separating 2-sphere is essential
on each side, by the above result.
Hence it is essential in $Q$,  
by a Mayer-Vietoris argument in $\widetilde{Q}$,
and so $\pi_2(Q)\not=0$.
\end{proof}

The equivariant chain complex $C_*(P;\mathbb{Z}[\pi])$
is chain homotopy equivalent to a finite projective chain complex,
since cohomology of $P$ is isomorphic to homology of the pair,
by Poincar\'e duality,
and so commutes with direct limits of coefficient modules.
Therefore $\pi$ is $FP_2$.
If $\partial{P}\not=\emptyset$ then $P$ is a retract of the double 
$DP=P\cup_{\partial{P}}P$, and so $\pi$ is a retract of $\pi_1(DP)$.
(In order to study $\pi$ we may assume that $\partial{P}$ 
has no $S^2$ components.)

\begin{lemma}
\label{indec nonor}
Let $(P,\partial{P})$ be a $PD_3$-pair such that $w=w_1(P)$ does not split.
If $G$ is an indecomposable factor of $\pi=\pi_1(M)$ then either $G$ 
is a $PD_3$-group,
or $G$ has one end and $c.d.G=2$, or $G$ is virtually free.
\end{lemma}

\begin{proof}
If $\partial{P}$ is empty then indecomposable factors of $\pi$ 
which are not virtually free are $PD_3$-groups \cite[Theorem 14]{Cr00}
(see also \cite[Theorem 4.8]{Hi}),
if the pair is orientable, and by \cite[Theorem 7.10]{Hi} otherwise.
If $\partial{P}\not=\emptyset$ then $\pi$ is a retract of $\pi(DP)$.
A splitting of $w_1(DP)$ would induce a splitting of $w$,
and so  $w_1(DP)$ does not split.
The lemma now follows from the Kurosh Subgroup Theorem,
as indecomposable factors of $\pi$ which are not virtually free
must be conjugate to subgroups of factors of $\pi_1(DP)$ which are $PD_3$-groups.
Thus if $G$ is indecomposable and not virtually free it is either 
a $PD_3$-group or has one end and $c.d.G=2$.
\end{proof}

If $w$ splits, 
a similar reduction to the absolute case and \cite[Theorem 7.10]{Hi} 
shows that $\pi$ may have indecomposable factors with infinitely many ends,
but the intersection of any non-orientable factor with $\pi^+=\mathrm{Ker}(w)$ 
is torsion free.
In general, $\pi$ is $vFP$ and $v.c.d.\pi\leq3$. 

If $\pi$ is a $PD_3$-group then the components of $\partial{P}$ 
are copies of $S^2$,
while if $c.d.\pi=2$ and $\pi$ has one end then 
$\partial{P}$ has at least one aspherical component,
and the peripheral system is $\pi_1$-injective \cite[Lemma 3.1]{Hi}.

Let $E(\pi)=H^1(\pi;\mathbb{Z}[\pi])$ be the {\it end module\/} of $\pi$,
and let $\Pi=\pi_2(P)=H_2(P;\mathbb{Z}[\pi])$.
The end module is naturally a right $\mathbb{Z}[\pi]$-module,
and the conjugate dual $\overline{E(\pi)}$ is isomorphic to 
$H_2(P,\partial{P};\mathbb{Z}[\pi])$, 
by Poincar\'e duality.
The interaction of $P$ and $\partial{P}$ are largely reflected in the exact sequence
\[
0\to{H_2(\partial{P};\mathbb{Z}[\pi])}\to\Pi\to\overline{E(\pi)}
\to{H_1(\partial{P};\mathbb{Z}[\pi])}\to0.
\]
derived from the exact sequence of homology for the pair,
with coefficients $\mathbb{Z}[\pi]$.
The group $H_1(\partial{P};\mathbb{Z}[\pi])$ is determined by the peripheral system,
and is 0 if the peripheral system is $\pi_1$-injective.
In general, it is a direct sum of terms corresponding to the 
compressible aspherical boundary components.
The group $H_2(\partial{P};\mathbb{Z}[\pi])$ is 0 if the pair has aspherical boundary
\cite[Lemma 3.1]{Hi}.

\begin{lemma}
\label{end module}
Let $G$ be a finitely generated group such that $G=*_{i=1}^mG_i*F(n)$, 
where $G_i$ has one end for $i\leq{m}$ and $m>0$.
Then $E(G)\cong\mathbb{Z}[G]^{m+n-1}$ as a right $\mathbb{Z}[G]$-module. 
\end{lemma}

\begin{proof}
The group $G$ is the fundamental group of a graph of groups whose underlying graph 
has $m$ vertices and $m+n-1$ edges, 
the vertex groups being the factors $G_i$ and the edge groups all being trivial.
The lemma follows immediately from the Chiswell Mayer-Vietoris sequence
for such graphs of groups.
\end{proof}

The hypothesis $m>0$ is necessary,
since $E(F(n))$ has a short free resolution with $n$ generators and one relator, 
but has no free direct summand.

We shall use this lemma with Poincar\'e duality to determine the rank 
of the maximal free factor of $\pi$ in terms of peripheral data.

\section{ estimating the maximal free factor}

Let $\kappa:S=\pi_1(Y)\to{G}$ be a geometric homomorphism.
We may assume that the associated family $\Phi$ of disjoint, 
2-sided simple closed curves on $Y$ is minimal, and shall then say that
$\Phi$ is a {\it geometric basis\/} for $\mathrm{Ker}(\kappa)$.
Let $r=r(\kappa)$ be the number of non-separating curves in $\Phi$.
Then $\kappa(S)\leq{G}$ is a free product $\kappa(S)\cong(*_{p=1}^aS_p)*F(r)$, 
where the $S_p$s are $PD_2$-groups or copies of $\mathbb{Z}/2\mathbb{Z}$ and 
$\Sigma_{p=1}^a\beta_1(S_p)=\beta_1(S)-2r(\kappa)$.
Moreover, $|\Phi|=r(\kappa)$, if $a=0$, and $|\Phi|=a+r(\kappa)-1$ otherwise. 
The  homomorphism $\kappa$ is  torsion free geometric if
and only if  no curve in $\Phi$ bounds a M\"obius band in $Y$.
Stallings' method of ``binding ties" \cite{St65} may be used to show that 
if $S\to{A*B}$ is an epimorphism then $Y$ decomposes accordingly 
as a connected sum. 
This together with the hopficity of $PD_2$-groups implies that a homomorphism 
$\kappa:S\to {G}$ is geometric if and only if
$\kappa(S)\cong(*_{p=1}^aS_p)*F(r)$, 
for some finite set $\{S_p|1\leq{p}\leq{a}\}$ as above and 
$r=\frac12(\beta_1(S)-\Sigma_{p=1}^a\beta_1(S_p))$.

\begin{lemma}
\label{short free}
Let $Y$ be an aspherical closed surface and $\kappa:S=\pi_1(Y)\to{B}$ 
a geometric homomorphism with geometric basis $\Phi$.
Assume that $\kappa$ is an epimorphism and $B$ is torsion free.
Then $H_1(S;\mathbb{Z}[B])$ has a short free resolution, 
$\mathbb{F}_2\otimes_BH_1(S;\mathbb{Z}[B])$ has dimension $|\Phi|$,
and $Tor_1^B(\mathbb{F}_2,H_1(S;\mathbb{Z}[B]))=\mathbb{F}_2$ if $B$ is free and is 
$0$ otherwise.
\end{lemma}

\begin{proof}
We may assume that $\Phi$ is non-empty.
Then $B\cong(*_{i=1}^aS_i)*F(r)$, where the $S_i$s are $PD_2$-groups,
and so there is a finite 2-dimensional $K(B,1)$ complex,
with one 0-cell, $2g-r$ 1-cells and $a$ 2-cells,
where $g=\frac12\beta_1(S)$.
The submodule $Z_1$ of 1-cycles in the chain complex $C_*(S;\mathbb{Z}[B])$ 
is a finitely generated stably free $\mathbb{Z}[B]$-module,
of stable rank $a+r$, 
by a Schanuel's Lemma argument, since $c.d.B\leq2$.
Since $B$ is infinite, $H_2(S;\mathbb{Z}[B])=0$.
Hence $H_1(S;\mathbb{Z}[B])$ has a short projective resolution
\[
0\to\mathbb{Z}[B]\to{Z_1}\to{H_1(S;\mathbb{Z}[B])}\to0.
\]
The 5-term exact sequence of low degree from the LHS homology spectral sequence 
for $S$ as an extension of $B$ by $K=\mathrm{Ker}(\kappa)$ is
\[
H_2(S;\mathbb{F}_2)\to{H_2(B;\mathbb{F}_2)}\to{H_0(B;H_1(K;\mathbb{F}_2))}\to
{H_1(S;\mathbb{F}_2)}\to{H_1(B;\mathbb{F}_2)}\to0.
\]
The central term is 
$H_0(B;H_1(K;\mathbb{F}_2))\cong\mathbb{F}_2\otimes_BH_1(S;\mathbb{Z}[B])$.

If $B$ is a free group then $H_2(B;\mathbb{F}_2)=0$,
$a=0$ and $r=g$, and so $H_0(B;H_1(K;\mathbb{F}_2))$ has dimension $r$.
If $a>0$ then $H_2(S;\mathbb{F}_2)$ maps injectively to 
$H_2(B;\mathbb{F}_2)\cong\mathbb{F}_2^a$,
since the natural epimorphisms from $S$ to the factors $S_i$ have degree 1,
and so $H_0(B;H_1(K;\mathbb{F}_2))$ has dimension $a+r-1$.
Since $|\Phi|=r$ if $B$ is a free group and $|\Phi|=a+r-1$ otherwise,
this proves the second assertion.

The final assertion follows on applying the tensor product 
$\mathbb{F}_2\otimes_B-$ to the above resolution of $H_1(S;\mathbb{Z}[B]))$ 
and using the fact that $Z_1$ is stably free of stable rank $a+r$.
\end{proof}

Note also that $Tor_i^B(\mathbb{F}_2,H_1(S;\mathbb{Z}[B]))=0$ for $i>1$.

Let $(G,\{\kappa_j|j\in{J}\})$ be a geometric group system such that 
$G\cong(*_{i\in{I}}G_i)*W$, 
where the factors $G_i$ each have one end and $W$ is virtually free,
and such that $\mathrm{Im}(\kappa_j)$ is torsion free, for all $j\in{J}$.
Let $\Gamma(G,\{\kappa_j\})$ be the bipartite graph with 
vertex set $I\sqcup{J}$ and an edge from $j\in{J}$ to $i\in{I}$ 
for each indecomposable factor of $\mathrm{Im}(\kappa_j)$ 
which is a $PD_2$-group and is conjugate to a subgroup of $G_i$.
This graph shall play the role of a group theoretic approximation to a template for the assembly of a $PD_3$-pair from simpler pieces by connected sums.

If $G$ is virtually free then $I$ is empty and so $\Gamma(G,\{\kappa_j\})$
has no edges.

\begin{lemma}
\label{min gen}
Let $(P,\partial{P})$ be a $PD_3$-pair with non-empty, 
aspherical boundary and peripheral system $(\pi,\{\kappa_j|j\in{J}\})$, 
and such that $\pi\cong\sigma*F(n)$,
where $\sigma\not=1$ and has no nontrivial virtually free factor.
Suppose also that $B_j=\mathrm{Im}(\kappa_j)$ is a free product 
of $PD_2$-groups, for $j\in{J}$.
Then $n\geq1-\chi(\Gamma(\pi,\{\kappa_j\}))$.
The pair is aspherical if and only if 
$c.d.\pi=2$ and $n=1-\chi(\Gamma(\pi,\{\kappa_j\}))$. 
\end{lemma}

\begin{proof}
Since $\sigma$ has no virtually free factors it is a free product of groups
with one end, by Lemma \ref{indec nonor}, and so we may apply 
Lemmas \ref{end module} and \ref{short free}.
Let $L=H_1(\partial{P};\mathbb{Z}[\pi])$ and 
$L_j=H_1(S_j;\mathbb{Z}[\pi])$, for $j\in{J}$.
Then $L=\oplus_{j\in{J}}L_j$.
The exact sequence relating $\Pi=\pi_2(P)$ to $\overline{E(\pi)}$
given above reduces to a short exact sequence
\[
0\to\Pi\to\overline{E(\pi)}\to{L}\to0,
\]
since the pair has aspherical boundary.
Let $m$ be the number of one-ended factors of $\pi$.
An application of Schanuel's Lemma shows that $\Pi$ is projective,
since $L$ has a short free resolution,
by Lemma \ref{short free},
and $\overline{E(\pi)}=\mathbb{Z}[\pi]^{m+n-1}$,
by Lemma \ref{end module}.

Applying the tensor product $\mathbb{F}_2\otimes_\pi-$, 
we get a sequence
\[
0\to\mathbb{F}_2\otimes_\pi\Pi\to
\mathbb{F}_2\otimes_\pi\overline{E(\pi)}\to\mathbb{F}_2\otimes_\pi{L}\to0,
\]
since $Tor_1^\pi(\mathbb{F}_2,L)=0$, by Lemma \ref{short free}.

For each $j\in{J}$,
let $\Phi_j$ be a geometric basis for $\kappa_j$,
and let $s_j=|\Phi_j|$.
Then $B_j=\mathrm{Im}(\kappa_j)$ has $s_j+1$ factors.
Since $L_j\cong\mathbb{Z}[\pi]\otimes_{B_j}H_1(S_j;\mathbb{Z}[B_j])$,
we have
\[
\mathbb{F}_2\otimes_\pi{L_j}=
\mathbb{F}_2\otimes_\pi\mathbb{Z}[\pi]\otimes_{B_j}H_1(S_j;\mathbb{Z}[B_j])=
\mathbb{F}_2\otimes_{B_j}H_1(S_j;\mathbb{Z}[B_j]).
\]
Hence 
$\dim\,\mathbb{F}_2\otimes_\pi{L}=\Sigma_{j\in{J}}|\Phi_j|=\Sigma_{j\in{J}}s_j$,
by Lemma \ref{short free}.

Since $\mathbb{F}_2\otimes_\pi\overline{E(\pi)}$ is an extension of 
$\mathbb{F}_2\otimes_\pi{L}$ by $\mathbb{F}_2\otimes_\pi\Pi$,
we see that 
\[
m+n-1=\dim\,\mathbb{F}_2\otimes_\pi\Pi+\Sigma_{j\in{J}}s_j.
\]
Since $\Gamma(\pi,\{\kappa_j\})$ has $m+|J|$ vertices and 
$\Sigma_{j\in{J}}(s_j+1)$ edges,
it has Euler characteristic 
$m+|J|-\Sigma_{j\in{J}}(s_j+1)=m-\Sigma_{j\in{J}}s_j$,
and so 
\[
n=1+\dim\,\mathbb{F}_2\otimes_\pi\Pi-\chi(\Gamma(\pi,\{\kappa_j\}))
\geq1-\chi(\Gamma(\pi,\{\kappa_j\})).
\]
If $P$ is aspherical then $\Pi=0$ and so we have equality.
We then have $c.d.\pi=2$, since $\partial{P}$ is non-empty and $\pi$ is not free.
Conversely, if $n=1-\chi(\Gamma(\pi,\{\kappa_j\}))$ then
$\dim\,\mathbb{F}_2\otimes_\pi\Pi=0$.
If we also have $c.d.\pi=2$ then the projective module $\Pi$ is 0 \cite{Ec86},
and so $P$ is aspherical.
\end{proof}

If $P$ is aspherical  and $n=0$ then  $\chi(\Gamma(\pi,\{\kappa_j\}))=1$.
If, moreover, 
$(P,\partial{P})=\natural^r(P_i,\partial{P_i})$ where $\pi_1(P_i)$ has one end,
for all $i$, 
then $\Gamma(\pi,\{\kappa_j\})$ is connected,
and so it is a tree. 
The converse follows from Theorem \ref{asph boundary} below:
if $\pi$ has no nontrivial virtually free factor and $\Gamma(\pi,\{\kappa_j\})$ 
is a tree then $(P,\partial{P})$ is such a boundary-connected sum.

When $c.d.\pi=2$ the spectral sequence of the universal cover 
$\widetilde{P}\to{P}$ gives another exact sequence
\[
0\to\mathbb{F}_2\otimes_\pi\Pi\to
{H_2(P;\mathbb{F}_2)}\to{H_2(\pi;\mathbb{F}_2)}\to0,
\]
and so 
$\dim\,\mathbb{F}_2\otimes_\pi\Pi=\beta_2(P)-\beta_2(\pi)=\chi(P)-\chi(\pi)$.
If we could show that $\chi(P)-\chi(\pi)\geq{b-1}$, 
where $b=\beta_0(\Gamma(\pi,\{\kappa_j\}))$
then the condition in Theorem \ref{asph boundary}
below that there be a free factor of sufficiently large rank 
would be necessary.
It clearly holds if $b=1$, or if $P$ is a connected sum of $b$ such
pairs with non-empty boundary.
(Note also that forming connected sums with copies of $S^2\times{S^1}$
or $S^2\tilde\times{S^1}$ changes both $n$ and $\chi(\pi)$, 
but does not change the boundary $\partial{P}$,
the graph $\Gamma(\pi,\{\kappa_j\})$, the factor $\sigma$
or the sum $n+\chi(\pi)=\chi(\sigma)$.)

We show next that free factors of peripheral subgroups contribute
to free factors of $\pi$.

\begin{lemma}
\label{restriction}
Let $(P,\partial{P})$ be a $PD_3$-pair, and let $\pi=\pi_1(P)$.
Let $S$ be an aspherical boundary component, 
and let $B$ be the image of $\pi_1(S)$ in $\pi$. 
Then restriction maps $H^1(\pi;\mathbb{Z}[\pi])$ onto
$H^1(B;\mathbb{Z}[\pi])$.
\end{lemma}

\begin{proof}
Since $H^2(P,\partial{P};\mathbb{Z}[\pi])=H_1(P;\mathbb{Z}[\pi])=0$,
restriction maps the end module
$H^1(P;\mathbb{Z}[\pi])=H^1(\pi;\mathbb{Z}[\pi])$ onto
$H^1(\partial{P};\mathbb{Z}[\pi])$.
The projection of $H^1(\partial{P};\mathbb{Z}[\pi])$ onto its summand 
$H^1(S;\mathbb{Z}[\pi])$ factors through $H^1(B;\mathbb{Z}[\pi])$, 
which is a subgroup of $H^1(S;\mathbb{Z}[\pi])$,
since $B$ is a quotient of $\pi_1(S)$.
Hence restriction maps $H^1(\pi;\mathbb{Z}[\pi])$ onto
$H^1(B;\mathbb{Z}[\pi])$.
\end{proof}

\begin{lemma}
\label{free factor}
Let $B<G$ be groups such that restriction maps $H^1(G;\mathbb{Z}[G])$ onto
$H^1(B;\mathbb{Z}[G])$. 
If $b\in{B}$ generates a free factor of $B$ then its image in
$G$ generates a free factor of $G$.
\end{lemma}

\begin{proof}
The first cohomology group $H^1(G;M)$ of $G$ with coefficients $M$ is 
the quotient of the group of $M$ valued derivations $Der(G;M)$ by the principal
derivations $Pr(G;M)$ \cite[Exercise III.1.2]{Br}.
Since restriction clearly maps $Pr(G;M)$ onto $Pr(B;M)$, for any $M$,
the hypothesis implies that restriction maps $Der(G;\mathbb{Z}[G])$ onto 
$Der(B;\mathbb{Z}[G])$.
If $b\in{B}$ generates a free factor of $B$ then there is a derivation 
$\delta:B\to\mathbb{Z}[B]$ such that $\delta(b)=1$ \cite[Corollary IV.5.3]{DD}.
This may be viewed as a derivation with values in $\mathbb{Z}[G]$,
and so is the restriction of a derivation $\delta_G:G\to\mathbb{Z}[G]$.
A second application of \cite[Corollary IV.5.3]{DD} now shows that $b$
generates a free factor of $G$, since $\delta_G(b)=\delta(b)=1$.
\end{proof}

An immediate consequence of these two lemmas is that if $(P,\partial{P})$
is a peripherally torsion free $PD_3$-pair such that $\pi=\pi_1(P)$ has 
no free factors then the image of $\pi_1(Y_j)$ in $\pi$ is a free product 
of $PD_2$-groups, for each boundary component $Y_j$.

\section{the decomposition theorems}

There are two distinct notions of connected sum for $PD_3$-pairs, 
corresponding to the interior and boundary connected sums of manifolds
\cite{Bl}.
The peripheral system of an interior connected sum is the free product of
the peripheral systems of the factors, 
while that of the boundary connected is slightly more complicated.
When we use the term ``connected sum" without further qualification, 
we shall allow both possibilities.
In each case, the fundamental group of the ambient space is a free product.

The sum of two geometric group systems $(G_1,\mathcal{K}_1)$
and $(G_2,\mathcal{K}_2)$ is the geometric group system
$(G_1,\mathcal{K}_1)\sharp(G_2,\mathcal{K}_2)=
(G_1*G_2,\mathcal{K}_1\sqcup\mathcal{K}_2)$.
A geometric group system {\it splits sharply\/} if it is isomorphic to 
such a sum.

Suppose that for some $\kappa:S\to{G}$ in $\mathcal{K}$ 
there is a non-trivial element $\gamma\in\mathrm{Ker}(\kappa)$
which is represented by a separating simple closed curve $C$ on $Y=K(S,1)$.
Then $Y/C\simeq{Y_1}\vee{Y_2}$, 
where $Y_1$ and $Y_2$ are closed surfaces,
and $\kappa$ factors through 
$S/\langle\langle\gamma\rangle\rangle\cong{S_1*S_2}$,
where $S_i=\pi_1(Y_i)$, for $i=1,2$.
Let $\kappa(i):S_i\to{G}$ be the induced homomorphisms,
and let $\mathcal{K}^\natural$ be the geometric system 
obtained by replacing $\kappa$ by the pair $\kappa(1), \kappa(2)$.
If $(G,\mathcal{K}^\natural)$ splits sharply 
(as $(G_1,\mathcal{K}_1)\sharp(G_2,\mathcal{K}_2)$, say)
we shall say that
$(G,\mathcal{K})$ {\it splits as a boundary sum}.
We may refer to $(G_1,\mathcal{K}_1)$ and 
$(G_2,\mathcal{K}_2)$ as components of $(G,\mathcal{K})$,
for either notion of splitting.

Let $(G,\mathcal{K})$ be a geometric group system, 
$w:G\to\mathbb{Z}^\times$  an orientation character and
$\mu$ a class in $H_3(G,\mathcal{K};\mathbb{Z}^w)$.
If $\gamma\in {S}$ is represented by a separating simple closed curve 
then $w(\gamma)=1$.
Hence if $\theta:G_1*G_2\to{G}$ is an isomorphism underlying a splitting 
(of either kind) then the restrictions of $w\circ\theta$ to the factors define orientation characters for the component geometric systems.
Similarly,  $\theta$ induces a direct sum splitting 
\[
H_3(G,\mathcal{K};\mathbb{Z}^w)\cong
{H_3(G_1,\mathcal{K}_1;\mathbb{Z}^w)}\oplus
{H_3(G_2,\mathcal{K}_2;\mathbb{Z}^w)},
\]
and so $\mu$ determines classes $\mu_1$ and $\mu_2$ in the summands.

The Decomposition Theorems of \cite{Bl} are formulated
in terms of pairs with $\pi_1$-injective aspherical boundaries.
The  hypothesis of $\pi_1$-injectivity is only used (via the Realization Theorem)
to show that the factors of the peripheral system are 
the peripheral systems of $PD_3$-pairs, 
and her arguments establish the following assertion

 {\it If $\mu$ satisfies the boundary compatibility and Turaev conditions 
then so do $\mu_1$ and $\mu_2$.}\\
We shall invoke this as the ``essence of the Decomposition Theorems",
since the enunciations of these theorems in \cite{Bl} assumes $\pi_1$-injectivity.

There is another construction which adds a free factor to the
the fundamental group.
(We assume here that the ambient space $P$ is connected.)
If $\partial{P}$ is non-empty then we may add a 1-handle 
by identifying a pair of discs in components of $\partial{P}$ to get a new 
$PD_3$-pair $(Q,\partial{Q})$, with $\pi_1(Q)\cong\pi_1(P)*\mathbb{Z}$.
(The construction involves choices, but we shall not need to be more precise.)

If the discs are in the same component of $\partial{P}$
then $(Q,\partial{Q})$ is the boundary connected sum of $(P,\partial{P})$ 
with $(D^2\times{S^1},T)$ or $(D^2\tilde\times{S^1},Kb)$ 
(depending on the relative orientations of the discs).
The following lemma reflects this construction.

\begin{lemma}
\label{trivial handle}
Let $(G,\mathcal{K})$ be a geometric group system.
If $\gamma\in\mathrm{Ker}(\kappa)$ is represented 
by a non-separating simple closed curve on a surface $Y$,  
for some $\kappa\in\mathcal{K}$, then 
$(G,\mathcal{K})$ splits as a boundary sum with one component
$(\mathbb{Z}, \mathbb{Z}^2)$ or $(\mathbb{Z}, \pi_1(Kb))$.
\end{lemma}

\begin{proof}
We may assume that $\kappa:S=\pi_1(Y)\to{G}$.
If $C$ is a simple closed curve of $Y$ representing $\gamma$ 
then $Y\cong{Y_1}\sharp{U}$, where $U\cong{T}$ or $Kb$
(depending on $w(\gamma)$,
and $C$ is a meridian of $U$.
Hence $S/\langle\langle\gamma\rangle\rangle\cong{S_1}*\mathbb{Z}$,
where $S_1=\pi_1(Y_1)$ and the second factor is the image of $\pi_1(U)$.
This free factor represents a free factor of $G$, 
by Lemmas \ref{restriction} and \ref{free factor},
and so $G\cong{G_1}*\mathbb{Z}$.
It is easy to see that $(G,\mathcal{K})$ splits as the boundary sum of a geometric system $(G_1,\mathcal{K}_1)$ and one of
$(\mathbb{Z}, \mathbb{Z}^2)$ or $(\mathbb{Z}, \pi_1(Kb))$.
\end{proof}

There remains one further case,  
exemplified by 3-manifolds with handles connecting
distinct boundary components.
These are neither a connected sum nor a boundary connected sum.
This situation is among those covered by the next lemma,
which extends the arguments of \cite{Bl}.

\begin{lemma}
\label{1-handle loop}
Let $(G,\mathcal{K})$ be a geometric group system,
$w$ an orientation character and 
$\mu\in{H_3(G,\mathcal{K};\mathbb{Z}^w)}$ a class which satisfies 
the boundary compatibility and Turaev conditions. 
Suppose that $G\cong{H}*F(r)$, 
where $H$ has one end and $F(r)$ is free of rank $r\geq1$,
and that the image of each $\kappa\in\mathcal{K}$ is a free product 
of $PD_2$-groups.
The inclusions of these factors determine a geometric group system 
$(H,\mathcal{K}_H)$, and the image of $\mu$ in $H_3(H,\mathcal{K}_H;\mathbb{Z}^w)$ satisfies the boundary compatibility and Turaev conditions. 
\end{lemma}

\begin{proof}
We may assume that $\mathcal{K}=\{\kappa_j|j\in{J}\}$,
where $\kappa_j:S_j\to{G}$ for all $j\in{J}$.
Let $Y_j$ be an aspherical closed surface with $\pi_1(Y_j)\cong{S_j}$
and let $\Phi_j$ be a geometric basis for $\mathrm{Ker}(\kappa_j)$, 
for each $j\in{J}$.
Let $V_j$ be the 2-complex obtained by adjoining one 2-cell to $Y_j$ 
along each curve in $\Phi_j$.
Then $V_j$ is homotopy equivalent to a wedge of aspherical
closed surfaces $Y_{jk}$.
(This uses the assumption that $\mathrm{Im}(\kappa_j)$ is a free product 
of $PD_2$-groups.)
We may choose disjoint discs on each $Y_{jk}$, 
so that after identifying discs in pairs appropriately we recover $V_j$.
Let $\kappa_{jk}:S_{jk}=\pi_1(Y_{jk})\to{G}$
be the induced monomorphism, for each index pair $jk$.
The images of these monomorphisms are each conjugate into the factor $H$, by the Kurosh subgroup theorem,
and so after replacing each $\kappa_{jk}$ by a conjugate homomorphism
we obtain a geometric system $(H,\mathcal{K}_H)$ in which all the homomorphisms in $\mathcal{K}_H$ are injective.

For each index pair $jk$ choose a map from $Y_{jk}$ to $K(H,1)$ which 
induces $\kappa_{jk}$ (up to conjugacy) and sends the chosen discs 
to the base-point.
Let $K_H\simeq{K(H,1)}$ be the mapping cylinder of the disjoint union 
of these maps,
and let $U$ be the quotient of $K_H$ obtained by identifying 
the chosen pairs of discs in $\amalg{Y_{jk}}$.
Then $U\simeq{K(H,1)}\vee{K(F(s),1)}$, for some $s\geq0$.
Let $K=(U\cup{s}.e^2)\vee{F(r)}$, 
where the 2-cells are attached along a basis 
for the free factor $F(s)$.
Then $K\simeq{K(G,1)}$.
Let $W$ be the image of $\amalg_{j\in{J}}{Y_{jk}}$ in $K$. 
Then $W=Y\cup{n}e^2$, where $Y=\amalg{Y_j}$ and 
$n=|\cup_{j\in{J}}\Phi_j|$.
Note also that $W$ may be obtained (up to homotopy) 
by attaching 1-cells to the disjoint union $\amalg{Y_{jk}}$.

We shall compare the pairs $(K,Y)$ and $(K_H,\amalg{Y_{jk}})$
by means of two intermediary pairs.
Since $Y=\amalg_{j\in{J}}Y_j$ is a subcomplex of $W$, 
there is an inclusion of pairs $(K,Y)\to(K,W)$.
Since $W$ may be obtained from $Y$ by attaching 2-cells, 
which represent relative 2-cycles for $(K,Y)$, 
we see that 
$H_3(K,Y;\mathbb{Z}^w)\cong{H_3(K,W;\mathbb{Z}^w)}$.
Since $C_q(K,Y;\mathbb{Z}[G])\cong{C_q(K,W;\mathbb{Z}[G])}$ for all 
$q\not=2$,
while $C_2(K,Y;\mathbb{Z}[G])\cong
{C_2(K,W;\mathbb{Z}[G])}\oplus\mathbb{Z}[G]^n$, we have
\[
F^2(C_*(K,Y;\mathbb{Z}[G]))\cong
{F^2(C_*(K,W;\mathbb{Z}[G]))}\oplus\mathbb{Z}[G]^n.
\]
The natural map from $K^*=K_H\vee{F(r,1)}$ to $K$ defines a homotopy equivalence $(K^*,W^*)\simeq(K,W)$, 
where $W^*$ is the union of $\amalg{Y_{jk}}$ with suitable paths in  $K_H$ 
to the wedge point. 
Since $W^*$ may be obtained from 
$\amalg{Y_{jk}}$ by attaching 1-cells, 
\[
H_3(K_H,\amalg{Y_{jk}};\mathbb{Z}^w)\cong
{H_3(K^*,W^*;\mathbb{Z}^w)}\cong
{H_3(K,Y;\mathbb{Z}^w)}.
\]
Hence
$H_3(H;\mathcal{K}_H;\mathbb{Z}^w)=
H_3(K_H,\amalg{Y_{jk}};\mathbb{Z}^w)\cong
{H_3(G,\mathcal{K};\mathbb{Z}^w)}$,
and so $\mu$ determines a class 
$\mu_H\in{H_3(H,\mathcal{K}_H;\mathbb{Z}^w)}$.
The projections of each surface $Y_j$ onto the surfaces $Y_{jk}$ 
are all degree 1 maps. 
On comparing the long exact sequences of homology for 
$(K_H,\amalg{Y_{jk}})$ and $(K,W)$,
we see that since $\mu$ satisfies the boundary compatibility condition then
so does $\mu_H$.

Every relative 1-cochain for $(W^*,\amalg{Y_{jk}})$ is a 1-cocycle,
and extends to a 1-cocycle for  $(K^*,\amalg{Y_{jk}})$.
Hence $F^2(C_*(K^*,\amalg{Y_{jk}};\mathbb{Z}[G]))$
is isomorphic to $F^2(C_*(K^*,W^*;\mathbb{Z}[G]))$.
This is in turn stably isomorphic to $F^2(C_*(K,W;\mathbb{Z}[G]))$,
and thus to $F^2(C_*(K,Y;\mathbb{Z}[G]))$.
Hence
\[
F^2(C_*(K,Y;\mathbb{Z}[G]))\oplus\mathbb{Z}[G]^a\cong
F^2(C_*(K^*,\amalg{Y_{jk}};\mathbb{Z}[G]))\oplus\mathbb{Z}[G]^b,
\]
for some $a,b\geq0$.

Let $D_*=C_*(K_H,\amalg{Y_{jk}};\mathbb{Z}[H])$.
Let $\alpha$ be the change of coefficients functor 
$\mathbb{Z}[G]\otimes_{\mathbb{Z}[H]}-$, and let $\beta$
be the left inverse induced by the retraction of $G$ onto $H$.
Then  $I(G)=\alpha{I(H)}$ and 
$F^2(C_*(K^*,\amalg{Y_{jk}};\mathbb{Z}[G]))\cong
\alpha{F^2(D_*)}$.
Hence
\[
F^2(C_*(K,Y;\mathbb{Z}[G]))\oplus\mathbb{Z}[G]^a\cong
\alpha{F^2(D_*)}\oplus\mathbb{Z}[G]^b.
\]
Let $f:F^2(D_*)\to{I(H)}$ be a representative of $\nu_{D_*,2}(\mu_H)$.
Then $\nu_{C_*,2}(\mu)$ is represented by the homomorphism 
$
\alpha{f}:\alpha{F^2(D_*)}\to\alpha{I(H)}.
$
Since $\nu_{C_*,2}(\mu)$ is a projective homotopy equivalence
there are finitely generated projective modules $L$ and $M$ 
and a homomorphism $h$ such that the following diagram commutes
\begin{equation*}
\begin{CD}
\alpha{F^2(D_*)}@>\alpha{f}>>\alpha{I(H)}\\
@VVV  @VVV\\
\alpha{F^2(D_*)}\oplus{L}@>h>>\alpha{I(H)}\oplus{M}.
\end{CD}
\end{equation*}
We apply the functor $\beta$.
Clearly $\beta{L}$ and $\beta{M}$ are
finitely generated projective $\mathbb{Z}[H]$-modules.
We obtain a commutative diagram
\begin{equation*}
\begin{CD}
F^2(D_*)@>f>>I(H)\\
@VVV  @VVV\\
F^2(D_*)\oplus\beta{L}
@>\beta{h}>>I(H)\oplus\beta{M},
\end{CD}
\end{equation*}
since $\beta\circ\alpha=id$.
Hence $f$ is the composite
\[
F^2(D_*)\to{F^2(D_*)}\oplus\beta{L}\cong
{I(H)}\oplus\beta{M}\to{I(H)},
\]
where the left- and right-hand maps are the obvious inclusion and projection,
respectively.
Hence $f$ is a projective homotopy equivalence,
and so $\mu_H$ satisfies the Turaev condition.
\end{proof}

We shall use Lemmas \ref{trivial handle} and \ref{1-handle loop}
and the essence of the Decomposition Theorems
inductively in Theorem \ref{asph boundary} below.

We may also add 2-handles.
 
\begin{lemma}
\label{2-handle}
Let $(P,\partial{P})$ be an aspherical $PD_3$-pair 
and let $(P_\gamma,\partial{P_\gamma})$ be the pair obtained by adding a 
$2$-handle along a $2$-sided simple closed curve $\gamma$ in a component 
$Y_j$ of $\partial{P}$.
If $(P_\gamma,\partial{P_\gamma})$ has aspherical boundary and the image of 
$\gamma$ generates a free factor of $\mathrm{Im}(\kappa_j)$ in $\pi=\pi_1(P)$
then $(P,\partial{P})\simeq(P_\gamma,\partial{P_\gamma})\natural(E,\partial{E})$,
where $E=D^2\times{S^1}$.
\end{lemma}

\begin{proof}
If the image of $\gamma$ generates a free factor of $B_j=\mathrm{Im}(\kappa_j)$
then $\pi\cong\rho*\mathbb{Z}$, by Lemmas \ref{restriction} and
\ref{free factor}.
We then may choose an isomorphism $\pi\cong\pi_1(P_\gamma\sharp{E})$ so that 
the peripheral systems of $(P,\partial{P})$ and
$(P_\gamma,\partial{P_\gamma})\natural(E,\partial{E})$ correspond.
If $(P_\gamma,\partial{P_\gamma})$ has aspherical boundary then
we may apply the second Decomposition Theorem of Bleile \cite{Bl}
to conclude that $(P,\partial{P})\simeq(P_\gamma,\partial{P_\gamma})\natural(E,\partial{E})$.
\end{proof}

When $P=D^2\times{S^1}$ and $\gamma$ is a longitude on $T=\partial{P}$ then $P_\gamma=D^3$. 
Thus even if $\partial{P}$ is aspherical $\partial{P_\gamma}$ may have an $S^2$ component.

Finally, we may add 3-handles, to cap off $S^2$ components of the boundary.
If $(P,\partial{P})$ is a $PD_3$-pair let $(\widehat{P},\partial\widehat{P})$ 
be the pair obtained by capping off each $S^2$ component of $\partial{P}$ 
with a copy of $D^3$.

\section{extending the realization theorem}

The Loop Theorem for 3-manifolds asserts that if the inclusion of 
a boundary component $F$ into a 3-manifold $M$ is not $\pi_1$-injective
then there is an essential simple loop on $F$ which bounds a disc in $M$.
Hence a 3-manifold with compressible boundary 
is either a boundary connected sum or has a 1-handle.
We shall say that {\it the Topological Loop Theorem holds\/} for all  
$PD_3$-pairs in some class if every such pair with compressible boundary
may be obtained from other pairs in that class by boundary connected sum 
or by adding a 1-handle.
Every orientable $PD_3$-pair may be obtained from
a boundary connected sum of $PD_3$-pairs with $\pi_1$-injective boundary by adding $1$-handles if and only if the Topological Loop Theorem holds
for all orientable $PD_3$-pairs.

Unfortunately, the Algebraic Loop Theorem is not yet strong enough 
to establish the Topological Loop Theorem for $PD_3$-pairs.
There remain difficulties related to free factors of the ambient group.

A $PD_3$-pair $(P,\partial{P})$ is {\it well-built\/} if it can be obtained 
by adding 1-handles to a connected sum of $PD_3$-complexes, 
aspherical $PD_3$-pairs $\{(P_i,\partial{P_i})|i\in{I}\}$ with $\pi_1$-injective boundaries and copies of $(D^2\times{S^1},T)$ or 
$(D^2\tilde\times{S^1},Kb)$.
If $\kappa:S\to{G}$ is a geometric homomorphism let $r(\kappa)$ be the number of non-separating curves in a geometric basis for $\kappa$.

\begin{theorem}
\label{asph boundary}
Let $G$ be a finitely presentable group and let $\{\kappa_j:S_j\to{G}|j\in{J}\}$ 
be a finite family of homomorphisms with domains $PD_2$-groups $S_j$. 
Let $w:G\to\mathbb{Z}^\times$ be a homomorphism which does not split
and let $\mu\in{H_3(G,\{\kappa_j\};\mathbb{Z}^w)}$.
If $[(G,\{\kappa_j\}),w,\mu]$ is the fundamental triple of a $PD_3$-pair
then
\begin{enumerate}
\item{}each indecomposable factor of $G$
with more than one end is virtually free;
\item$\kappa_j$ is torsion free geometric and $w\circ\kappa_j=w_1(S_j)$, 
for $j\in{J}$; 
\item{}the images of the free factors of the $\kappa_j(S_j)$s in $G$ 
generate a free factor of rank $r=\Sigma_{j\in{J}}r(\kappa_j)$;
and
\item$\mu$ satisfies the boundary compatibility and Turaev conditions.
\end{enumerate}
Conversely, if $[(G,\{\kappa_j\}),w,\mu]$ satisfies these conditions
and $G$ has a free factor of rank $r+s$,
where $s=\beta_1(\Gamma(G,\{\kappa_j\}))$,
then $[(G,\{\kappa_j\}),w,\mu]$ is the fundamental triple of a 
well-built $PD_3$-pair.
\end{theorem}
 
\begin{proof}
Let $(P,\partial{P})$ be a $PD_3$-pair such that $w=w_1(P)$ does not split.
Then the indecomposable factors of $\pi=\pi_1(P)$ are either one-ended
or virtually free, by Lemma \ref{indec nonor}.
Condition (2) holds by the Algebraic Loop Theorem.
(The peripheral homomorphisms are torsion free since $w$ does not split.)
Condition (3) holds as a consequence of Lemmas \ref{restriction} 
and \ref{free factor}.
The boundary compatibility and Turaev conditions are satisfied, 
as we saw in \S1.
Thus (1)--(4) are necessary conditions.

Suppose that these conditions hold.
We shall use Lemmas \ref{trivial handle} and \ref{1-handle loop} 
and the essence of the Decomposition Theorems
inductively, to replace the geometric system by simpler components
with fundamental classes which still satisfy the boundary  compatibility and Turaev conditions.
Reassembling the resulting pieces is then straightforward.

The first step is to use Lemma \ref{trivial handle}
repeatedly to arrange that $r=0$,
i.e., that $\mathrm{Im}(\kappa_j)$ is a free product of $PD_2$-groups
$S_{jk}$, for $j\in{J}$.
The ambient group $G$ then has a factorization 
$G\cong(*_{i\in{I}}G_i)*G_\omega*F(t)$, 
where $G_i$ has one end, for all $i\in{I}$, 
$G_\omega$ is virtually free,  
but has no $\mathbb{Z}$ free factor,  and $t\geq{s}$.
The indecomposable factors of $\mathrm{Im}(\kappa_j)$ 
are then $PD_2$-groups, 
and so are each conjugate into one of the factors $G_i$, for some $i\in{I}$,
by the Kurosh Subgroup Theorem.
Let $\mathcal{K}_i$ be the family of inclusions of the factors of
$\amalg_{j\in{J}}\mathrm{Im}(\kappa_j)$ which are conjugate to 
subgroups of $G_i$,  for $i\in{I}$.
(The analogous set $\mathcal{K}_\omega$ is empty, of course.)

The second step is to use the essence of the First Decomposition Theorem 
(sharp splitting) to split off the factor $G_\omega$ and to
reduce to the case when the graph $\Gamma(G,\{\kappa_j\})$ is connected.
We may also retain the condition $t\geq{s}$, 
since $\beta_1(\Gamma(G,\{\kappa_j\}))$ is the sum of the $\beta_1$s of the 
components of $\Gamma$.

The third step is to use the essence of the second Decomposition Theorem
(boundary sum splitting)
to reduce further to the case when $|I|=1$.
Finally we use Lemma \ref{1-handle loop} to reduce to the case when
all the homomorphisms in the geometric system are injective,
and so $s=0$.
At each stage, $\mu$ determines homology classes for the component systems which satisfy the  boundary compatability and Turaev conditions.

At the final stage,
each triple $[(G_i,\mathcal{K}_i),w|_{G_i},\mu_i]$ determines 
a $PD_3$-pair $(X_i,\partial{X_i})$ with aspherical boundary and 
$\pi_1$-injective peripheral system, by Bleile's theorem 
\cite[Theorem 1.2]{Bl}.
Note also that $\Gamma(G_i,\mathcal{K}_i)$ is a tree.
Hence if $B\in\mathcal{K}_k$ and $C\in\mathcal{K}_\ell$ are distinct factors 
of the image of $\kappa_j(S_j)$ then $\ell\not=k$.
We form the boundary connected sum of 
$(X_k,\partial{X_k})$ with $(X_\ell,\partial{X_\ell})$ 
along the corresponding boundary components, 
and repeat this process until we have a family of connected $PD_3$-pairs
(one for each component of $\Gamma(G,\{\kappa_j\})$).
We then add further 1-handles as needed,
to obtain the correct graph $\Gamma(G,\{\kappa_j\})$.
Similarly, we may realize $[G_\omega,w|_{G_\omega}, \mu_\omega]$ by a $PD_3$-complex $X_\omega$.
The (interior) connected sum of all such pairs 
with a closed 3-manifold $M$ with  fundamental group $F(t-s)$
gives a $PD_3$-pair with fundamental triple $[(G,\{\kappa_j\}),w,\mu]$.
(The 1-handles and the 3-manifold summand should be chosen compatibly with $w$.) It is clear that the construction gives a well-built pair.
\end{proof}

\begin{cor}
If $w$ does not split then $[(G,\{\kappa_j\}),w,\mu]$ is 
the fundamental triple of a well-built $PD_3$-pair if and only if
$G$ has a free factor of rank $r+s$ and conditions (1)--(4) hold.
\qed
\end{cor}

If $(P,\partial{P})$ is a $PD_3$-pair such that virtually free factors of 
$\pi_1(P)$ are free then one might expect to apply Lemma \ref{min gen} 
after the first step, 
in order to show that having a free factor of rank $r+s$ is always necessary.
Unfortunately,  we have been unable to show without this hypothesis
that the reductive steps above preserve the property of 
being realizable by a $PD_3$-pair.

Theorem \ref{asph boundary} could be reformulated as giving criteria
for a triple to be ``stably realizable", i.e.,  
realizable after replacing $G$ by $G*F(m)$ for some $m\geq0$, 
and each $\kappa_j$ by its composite with the inclusion of $G$ into $G*F(m)$.
Let {\it stable equivalence\/} of $PD_3$-pairs be the equivalence relation generated by taking connected sums with copies of $S^2\times{S^1}$.

\begin{cor} 
Let  $(P,\partial{P})$ be a $PD_3$-pair with aspherical boundary
and such that $w_1(P)$ does not split.
Then $(P,\partial{P})$ is stably equivalent to the boundary connected 
sum of $PD_3$-pairs with $\pi_1$-injective peripheral systems and 
copies of $(D^2\times{S^1},T)$ or $(D^2\tilde\times{S^1}, Kb)$.
\qed
\end{cor}

We may extend the Decomposition Theorems in a similar way,
after allowing stabilization.

An aspherical $PD_3$-pair is well-built if it can be obtained  by
boundary connected sums and adding 1-handles from a set of aspherical $PD_3$-pairs with $\pi_1$-injective aspherical boundaries and copies of
$(D^2\times{S^1},T)$ or $(D^2\tilde\times{S^1},Kb)$.
Every such pair is aspherical.
If $(\pi,\{\kappa_j\})$ is the peripheral system of such a
$PD_3$-pair then $\Gamma(\pi,\{\kappa_j\})$ is connected and 
$\pi\cong\pi\mathcal{G}$, 
where $\mathcal{G}$ is the graph of groups with underling graph 
$\Gamma(\pi,\{\kappa_j\})$, 
vertex groups $G_i$ and $B_j$ for $i\in{I}$ and $j\in{J}$, 
and an edge group $B_{ik}$ whenever $B_{ik}$ is a factor of $B_j$.

\begin{cor}
\label{sap+vf}
Let $(P,\partial{P})$ be a $PD_3$-pair with aspherical boundary 
and peripheral system $(\pi,\{\kappa_j\})$.
Then $(P,\partial{P})\simeq(Q,\partial{Q})\sharp(R,\partial{R})$,
where $(Q,\partial{Q})$ is a well-built aspherical pair with no
summand having free fundamental group
and $\pi_1(R)\cong{F(r)}$,  if and only if $\pi$ is torsion free,
$\Gamma(\pi,\{\kappa_j\})$ is connected and $\pi$ has a free factor of rank 
$r+\beta_1(\Gamma(\pi,\{\kappa_j\}))$.
\end{cor}

\begin{proof}
The first two conditions are clearly necessary, 
and the third follows from Lemma \ref{min gen}.
The converse follows from Theorem \ref{asph boundary}
and the Classification Theorem for pairs with aspherical boundary \cite{Bl}.
\end{proof}

It follows from Lemma \ref{null boundary}  that the pair constructed 
here is aspherical if and only if there is no such summand $R$ with $\pi_1(R)\not=1$.

Here are some more specialized corollaries. 

\begin{cor}
\label{btfree}
If $G$ is a free group then $[\{\kappa_j\},w,\mu]$ is the fundamental triple 
of a $PD_3$-pair if and only if 
$w\circ\kappa_j=w_1(S_j)$ for $j\in{J}$ and condition (3) holds.
\end{cor}

\begin{proof}
Since $G$ is free $H_i(G;\mathbb{Z}^w)=0$ for $i>1$,
and so 
$H_3(G;\{\kappa_j\};\mathbb{Z}^w\cong\oplus{H_2(S_j;\mathbb{Z}^w)}$.
Thus we may choose $\mu$ to satisfy the boundary compatibility condition.
The augmentation ideal $I(G)$ is free as a $\mathbb{Z}[G]$-module,
and so all homomorphisms into $I(G)$ are projectively homotopy equivalent.
Hence $\mu$ satisfies the Turaev conditions for trivial reasons.
\end{proof}

\begin{cor}
If $\chi(S_j)=0$ for all $j\in{J}$ then $[\{\kappa_j\},w,\mu]$ 
is the fundamental triple of a peripherally torsion free $PD_3$-pair 
if and only if conditions (1)--(4) hold.
\end{cor}

\begin{proof}
There are no separating essential simple closed curves on a torus,
and the only such curves on the Klein bottle bound M\"obius bands.
Thus $\beta_1(\Gamma(G,\{\kappa_j\}))=0$.
\end{proof}

\begin{cor}
If $w$ is an orientation character that does not split and
$\beta_1(\Gamma(G,\{\kappa_j\}))=0$ 
 then $[(G,\{\kappa_j\}),w,\mu]$ 
is the fundamental triple of a $PD_3$-pair 
if and only if conditions (1)--(4) hold.
\qed
\end{cor}

A handlebody is a boundary connected sum of copies of 
$D^2\times{S^1}$ and $D^2\tilde\times{S^1}$. 
Thus it is a 3-manifold with connected boundary and which has a handle decomposition with only 0- and 1-handles.

\begin{cor}
Let $(P,\partial{P})$ be  a $PD_3$-pair with $\pi_1(P)$ virtually free
and aspherical boundary.
If $w$ does not split then $(P,\partial{P})$ is the interior connected sum 
of a $PD_3$-complex and handlebodies.
\qed
\end{cor}

\section{independence of the conditions}

Theorem \ref{asph boundary} has three principal hypotheses: 
(a) boundary compatibility,  (b) the Turaev condition and (c) the ALT
condition, imposed by the Algebraic Loop Theorem.
It is natural to ask whether any of these are implied by the others.
The following simple examples (with trivial orientation character)
suggest that they are independent, 
in at least two of the three cases.

1) Let $G=\mathbb{Z}$ and $\kappa:T=\mathbb{Z}^2\to{G}$ be the trivial homomorphism.
The connecting homomorphism from $H_3(G,\kappa;\mathbb{Z})$ 
to $H_2(T;\mathbb{Z})$ is an isomorphism, 
and so the boundary compatibility condition holds for 
either generator of $H_3(G,\kappa;\mathbb{Z})$,
while the Turaev condition holds vacuously (see Corollary \ref{btfree}).
However this pair does not satisfy the ALT condition,
since any two disjoint simple closed curves on the torus are parallel.
Thus (a) and (b) together do not imply (c).

2) Let $H$ be the Higman superperfect group,  with presentation
\[
\langle{a,b,c,d}\mid{a^2=bab^{-1}},~b^2=cbc^{-1},~c^2=dcd^{-1},~d^2=ada^{-1}\rangle
\]
and let $G=H\times\mathbb{Z}$.
Let $\kappa$ be the inclusion of $T$ as the subgroup generated by 
$a$ and the second factor.
Then $H_i(G;\mathbb{Z})=0$ for $i>1$, since 
$c.d.H=2$ and $H_i(H;\mathbb{Z})=0$ for $i>0$.
Hence the connecting homomorphism from 
$H_3(G,\kappa;\mathbb{Z})$ to $H_2(T;\mathbb{Z})$ is an isomorphism, 
and so the boundary compatibility condition holds.
The ALT condition also holds, since $\kappa$ is injective.
If the Turaev condition also held then $(G,\kappa)$ would be 
the peripheral system of a $PD_3$-pair, 
by Bleile's Realization Theorem.
However, this is not possible \cite[Lemma 3.1]{Hi}.
Thus (a) and (c) together do not imply (b).

Another such example is given by the pair $(G,\kappa)$, 
where $G=\pi{K}$ is the group of a satellite knot $K$ and $\kappa$ 
is the inclusion of $T$ corresponding to an essential,  
non-boundary parallel torus.
Then $(G,\kappa)$ satisfies the boundary and ALT conditions,
but the Bieri-Eckmann Splitting Theorem \cite[Theorem 8.4]{BE}
may be applied to the double $D(G,T)=G*_TG$ to show that 
$(G,\kappa)$ is not the peripheral system of a $PD_3$-pair.

3) Whether the boundary condition is indispensable is less clear.
Let $G$ be a finitely presentable group with $c.d.G=2$.
Then $I(G)$ is not projective,  and so no stable isomorphism from a 
$\mathbb{Z}[G]$-module $M$ to $I(G)$ is projectively null-homotopic.
Hence if there is a $PD_2$-group $J$ and a monomorphism
$\kappa:J\to{G}$ such that $(G,\kappa)$ satisfies the Turaev condition for some $\mu$ then $H_3(G,\kappa;\mathbb{Z})\not=0$.
The condition $c.d.G=2$ also implies that $H_2(G;\mathbb{Z})$
is torsion free and $H_3(G;\mathbb{Z})=0$.
Hence in this case $H_3(G,\kappa;\mathbb{Z})\cong{H_2(J;\mathbb{Z})}$,
and so the boundary condition is implied by the other conditions.
To go further we need a better understanding of when the Turaev condition holds.

\section{pairs with $\pi$ indecomposable and virtually free}

The condition that $\pi$ is indecomposable as a free product is more restrictive
than the pair being indecomposable as a connected sum,
although these notions are equivalent in the absolute case \cite{Tu}.
The pair obtained from $RP^2\times([0,1],\{0,1\})$ 
by adding a 1-handle to connect the two boundary components
is perhaps the simplest counter-example.

\begin{lemma}
\label{vfree wnont}
Let $(P,\partial{P})$ be a $PD_3$-pair, 
and let $\pi=\pi_1(P)$ and $w=w_1(P)$.
If $\pi$ is virtually free and $V$ is an indecomposable factor of
$\pi$ such that $w|_V$ is nontrivial then 
$V\cong\mathbb{Z}\oplus\mathbb{Z}/2\mathbb{Z}$,
$\mathbb{Z}$ or $\mathbb{Z}/2\mathbb{Z}$.
\end{lemma}

\begin{proof}
We may assume that the boundary is $\pi_1$-injective \cite[Theorem 2]{Cr07}
and that $\partial{P}$ has no $S^2$ boundary components.
Then $\partial{P}$ is a union of copies of $RP^2$,
and so $\pi_1(DP)$ is virtually free.
The indecomposable summands of $DP$ are either orientable or copies of
$S^2\tilde\times{S^1}$ or $RP^2\times{S^1}$ \cite[Theorems 7.1 and 7.4]{Hi}.
The lemma follows, since $V$ is a retract of $\pi_1(DP)$ and $w|_V$ is nontrivial.
\end{proof}

\begin{theorem}
\label{vf wsplit}
Let $(P,\partial{P})$ be a $PD_3$-pair with no $S^2$ boundary components
and such that $\pi=\pi_1(P)$ is indecomposable and virtually free. Then
\begin{enumerate}
\item{}If $w=w_1(P)$ splits and $\pi$ is infinite then $\partial{P}=\emptyset$ 
and $P\simeq{RP^2}\times{S^1}$;
\item{}If $w$ does not split then either $\partial{P}=\emptyset$ or
$(P,\partial{P})\simeq(D^2\times{S^1},T)$ 
or $(D^2\tilde\times{S^1},Kb)$.  
\end{enumerate}
\end{theorem}

\begin{proof}
Suppose first that $w$ splits and $\pi$ is infinite.
Then $\pi\cong\mathbb{Z}\oplus\mathbb{Z}/2\mathbb{Z}$,
by Lemma \ref{vfree wnont}.

The element of order 2 in $\pi$ is orientation-reversing,
since it has infinite centralizer  \cite[Theorem 17]{Cr00},
and so $\pi^+=\pi_1(P^+)\cong\mathbb{Z}$.
A simple argument applying Schanuel's Lemma to the cellular chain complex 
$C_*(P^+;\mathbb{Z}[\mathbb{Z}])$ shows that $\pi_2(P^+)$ is stably free of 
rank $\chi(P^+)$ as a $\mathbb{Z}[\mathbb{Z}]$-module.
Moreover,  stably free $\mathbb{Z}[\mathbb{Z}]$-modules are in fact free.

On considering the exact sequence of homology for $(P,\partial{P})$ with coefficients 
$\mathbb{F}_2$, we see that $\beta_1(\partial{P})\leq4$.
The covering pair associated to the subgroup 
$n\mathbb{Z}\oplus\mathbb{Z}/2\mathbb{Z}$ satisfies the same bounds, 
for any $n\geq1$.
Since any $RP^2$ in $\partial{P}$ would have $n$ preimages in such a cover,
$\partial{P}$ can have no $RP^2$ boundary components.
Since $\chi(\partial{P})=2\chi(P)$, it follows that $\chi(P)\leq0$.
Hence $\chi(P^+)\leq0$.
Suppose that $\partial{P}$ is not empty.
Then $\partial{P}^+\not=\emptyset$, so $H_3(P^+;\mathbb{Z})=0$.
Hence $\chi(P^+)=0$, and so $\pi_2(P^+)=0$, since it is a free
$\mathbb{Z}[\pi^+]$-module of rank 0.
Therefore $P^+$ is aspherical.
Since $p$ is finite-dimensional and $\pi$ has torsion, this is a contradiction.
Therefore $\partial{P}$ is empty.
Hence $P\simeq{RP^2}\times{S^1}$ \cite{Wa}.

If $w$ does not split then $P$ has aspherical boundary.
It follows from condition (3) of Theorem \ref{asph boundary}
that $\beta_1(\partial{P})\leq1$, and if $\beta_1(\partial{P})=1$ then
$\pi\cong\mathbb{Z}$. 
Thus either $\partial{P}=\emptyset$ or $\pi\cong\mathbb{Z}$.
\end{proof}

We shall show that $RP^2\times([0,1],\{0,1\})$ is the only indecomposable,
non-orientable pair with no $S^2$ boundary components and finite
fundamental group in Theorem \ref{nonor finite} below.

\section{spherical boundary components}

In this section we shall show that $PD_3$-pairs with boundary having some
$S^2$ components but no $RP^2$ components may be classified  in terms of 
slightly different invariants.
Instead of using the image  of the fundamental class in the group homology,
we use the first $k$-invariant.
In the absolute case, the fundamental triple $[\pi,w,\mu]$ of
a $PD_3$-complex $P$ determines $P$ among other $PD_3$-complexes,
whereas (when $\pi=\pi_1(M)$ is infinite) the triple $[\pi,w,k_1]$ determines 
$P$ among 3-dimensional complexes with $H_3(\widetilde{P};\mathbb{Z})=0$.
There is not yet a useful characterization of the $k$-invariants
which are realized by $PD_3$-complexes (or $PD_3$-pairs)
with infinite fundamental group.

We shall assume throughout this section that $\pi$ is infinite.
(This assumption shall be repeated in the statements of results, for clarity.)
Then $H_3(P;\mathbb{Z}[\pi])=H_3(\widetilde{P};\mathbb{Z})=0$, 
and so the homotopy type of $P$ is determined by $\pi$, 
$\Pi=\pi_2(P)=H_2(P;\mathbb{Z}[\pi])$ and the orbit of the first $k$-invariant 
$k_1(P)\in{H^3(\pi;\Pi)}$
under the actions of $Aut(\pi)$ and $Aut_\pi(\Pi)$.
The homotopy type of the pair involves the peripheral system and the 
inclusions of the spherical components (meaning copies of $S^2$ and/or $RP^2$).
If $\pi$ is torsion free then it is of type $FP$ and $c.d.\pi\leq3$,
as observed in \S2 above.
Hence $\Pi$ is a finitely generated projective $\mathbb{Z}[\pi]$-module,
by a Schanuel's Lemma  argument, 
and $\mathbb{Z}\otimes_{\mathbb{Z}[\pi]}\Pi\cong\mathbb{Z}^{\chi(P)-\chi(\pi)}$.
(By the same argument, if $\pi$ is of type $FF$ then $\Pi$ is stably free.)

Let $\alpha_P:\Pi\to\overline{E(\pi)}$ 
be the composite of the Poincar\'e duality isomorphism 
$H_2(P,\partial{P};\mathbb{Z}[\pi])\cong
\overline{E(\pi)}=\overline{H^1(P;\mathbb{Z}[\pi])}$
with the natural homomorphism from $\Pi=H_2(P;\mathbb{Z}[\pi])$ to 
$H_2(P,\partial{P};\mathbb{Z}[\pi])$.
Let $m(P)$ be the number of $S^2$ components of $\partial{P}$.

\begin{theorem}
\label{S^2 components}
Let $(P,\partial{P})$ and $(Q,\partial{Q})$ be $PD_3$-pairs with 
peripheral systems $\{\kappa_j^P|j\in{J}\}$ and $\{\kappa_j^Q|j\in{J}\}$, 
respectively, and such that $\pi=\pi_1(P)\not=1$.
If $c.d.\pi\leq2$ then $(P,\partial{P})\simeq(Q,\partial{Q})$ if and only if
\begin{enumerate}
\item$m(P)=m(Q)$;
\item{}there are isomorphisms $\theta:\pi_1(P)\cong\pi_1(Q)$
and $\theta_j:S_j^P\to{S_j^Q}$ such that  $w_1(P)=w_1(Q)\circ\theta$ 
and $\theta\kappa_j^P$ is conjugate to $\kappa_j^Q\theta_j$, 
for all $j\in{J}$.
\end{enumerate}
In general, these conditions determine a $\mathbb{Z}[\pi]$-linear 
isomorphism $g:\pi_2(P)\to\theta^*\pi_2(Q)$
such that $\alpha_P=E(\theta)\alpha_Qg$.
Hence if $\partial{P}$ and $\partial{Q}$ have no $RP^2$ boundary components
then $(P,\partial{P})\simeq(Q,\partial{Q})$ if and only if (1) and (2) hold and
$\theta^*k_1(Q)=g_\#k_1(P)$,
up to the actions of $Aut(\pi)$ and $Aut(\pi_2(Q))$.
\end{theorem}

\begin{proof}
The conditions are necessary,
for if $F:(P,\partial{P})\to(Q,\partial{Q})$ is a homotopy equivalence of pairs 
then we may take $\theta=\pi_1(F)$, $\theta_j=\pi_1(F|_{Y_j})$ and $g=\pi_2(F)$.

Suppose that they hold. 
The Postnikov 2-stage $P_2(Q)$ may be constructed by adjoining cells 
of dimension $\geq4$ to $Q$, and isomorphisms $\theta$ and $g$ such that
$\theta^*k_1(Q)=g_\#k_1(P)$ determine a map from $P$ to $P_2(Q)$.
Since $P$ has dimension $\leq3$, 
we may assume that such a map factors through $Q$, 
and so we get a map $F:P\to{Q}$ such that $\pi_1(F)=\theta$ and $\pi_2(F)=g$.
Since $\pi$ is infinite,
$H_q(\widetilde{P};\mathbb{Z})=H_q(\widetilde{Q};\mathbb{Z})=0$ for all $q>2$.
Therefore $F$ is a homotopy equivalence, by the Hurewicz and Whitehead theorems.
We shall show that we may choose $g$ to be compatible with the inclusions 
of the boundary components.
It respects the aspherical boundary components, by hypothesis.
If this is also the case for the spherical components then 
we may assume that $F$ maps $\partial{P}$ into $\partial{Q}$, 
and so is a homotopy equivalence of pairs.

Let $A=\mathrm{Im}(\alpha_P)$ and 
$M=\mathrm{Ker}(\alpha_P)=H_2(\partial{P};\mathbb{Z}[\pi])$. 
Since $(P,\partial{P})$ is peripherally torsion free,
$\partial{P}$ has no $RP^2$ boundary components,
and so $M\cong\mathbb{Z}[\pi]^{m(P)}$,
with basis determined by the $S^2$ boundary components.

Suppose first that $\pi$ is torsion free.
Let $Y_j$ be an aspherical component of $\partial{P}$,
and let $B_j=\mathrm{Im}(\kappa_j)$.
Since $H_1(Y_j;\mathbb{Z}[B_j])$ has a short free resolution
as a left $\mathbb{Z}[B_j]$-module, by Lemma \ref{short free},
$H_1(Y_j;\mathbb{Z}[\pi])$ also has such a resolution
as a left $\mathbb{Z}[\pi]$-module.
Summing over all such components of $\partial{P}$,
we get a short exact sequence
\[
0\to\mathbb{Z}[\pi]^p\to\mathbb{Z}[\pi]^q\to{H_1(\partial{P};\mathbb{Z}[\pi])}\to0.
\]
Since $A$ is the kernel of the epimorphism from $\overline{E(\pi)}$ to 
$H_1(\partial{P};\mathbb{Z}[\pi])$, we have
$A\oplus\mathbb{Z}[\pi]^q\cong\overline{E(\pi)}\oplus\mathbb{Z}[\pi]^p$,
by Schanuel's Lemma.

If $\pi$ is not free  then $\overline{E(\pi)}$ is a finitely generated 
free module, by Lemma \ref{end module}.
Hence $A$ is projective, and so $\pi_2(P)\cong{M}\oplus{A}$.
We may choose an isomorphism $g:\pi_2(P)\to\pi_2(Q)$ which respects 
the inclusions of the boundary spheres,
since the automorphisms of $\pi_2(P)$ that preserve the projection to $A$
act transitively on the bases for $\mathrm{Ker}(\alpha_P)$.

Now suppose that $\pi\cong{F(r)}$ is free of rank $r$, for some $r>0$.
The projective $\mathbb{Z}[\pi]$-module $\pi_2(P)$ is then free,
since all projective $\mathbb{Z}[F(r)]$-modules are free.
(In fact it has rank $\chi(P)+r-1$,
and so $P\simeq\vee^rS^1\vee^{\chi(P)+r-1}S^2$,
but we shall not need this.)
In this case $\overline{E(\pi)}$ is not projective; it has a short free
presentation with $r$ generators and one relator.
Since $A$ is projectively stably isomorphic to $\overline{E(\pi)}$,
and thus is not projective,
$\pi_2(\partial{P})$ is not a direct summand of $\pi_2(P)$.
However, 
\[
Ext^1_{\mathbb{Z}[\pi]}(A,M)\cong{Ext^1_{\mathbb{Z}[\pi]}(\overline{E(\pi)},M)}
\cong{M/\mathcal{I}M}\cong\mathbb{Z}^m.
\]
The extension class is (up to sign) the diagonal element $(1,\dots,1)$,
since the image of $[\partial{P}]$ in $H_2(P;\mathbb{Z}^w)$ is 0.
Since the extension classes for $\pi_2(P)$ and $\pi_2(Q)$ correspond,
there is an isomorphism $g:\pi_2(P)\to\pi_2(Q)$ which carries the given basis for 
the image of $\pi_2(\partial{P})$ to the given basis for $\pi_2(\partial{Q})$,
and which induces the isomorphism of the quotients determined by duality 
and the isomorphism of the peripheral data.

We may extend these arguments to all $PD_3$-pairs having no $RP^2$ 
boundary components, 
as follows.
Let $\nu$ be a torsion free subgroup of finite index in $\pi$.
If $L$ is a $\mathbb{Z}[\pi]$-module let $L|_\nu$ be the $\mathbb{Z}[\nu]$-module
obtained by restriction of scalars.
Then there are natural isomorphisms
$Hom_{\mathbb{Z}[\pi]}(L,\mathbb{Z}[\pi])\cong
{Hom_{\mathbb{Z}[\nu]}(L|_\nu,\mathbb{Z}[\nu])}$,
for all $\mathbb{Z}[\pi]$-modules $L$.
Since restriction preserves exact sequences and carries projectives to projectives,
it follows that
$Ext^i_{\mathbb{Z}[\pi]}(L,\mathbb{Z}[\pi])\cong
{Ext^i_{\mathbb{Z}[\nu]}(L|_\nu,\mathbb{Z}[\nu])}$,
for all such modules $L$ and for all $i\geq0$.
Hence if $v.c.d.\pi=2$ or 3 then $Ext^1_{\mathbb{Z}[\pi]}(A,\mathbb{Z}[\pi])=0$,
while if $\pi$ is virtually free of rank $\geq1$ then 
$Ext^1_{\mathbb{Z}[\pi]}(A,\mathbb{Z}[\pi])\cong\mathbb{Z}$, 
and we may argue as before.

In each case, the homotopy type of $P$ is determined by $\pi$, 
$m(P)$ and $k_1(P)$,
while the homotopy type of the pair $(P,\partial{P})$ is determined by
the peripheral system, $m(P)$ and $k_1(P)$.
Finally, if $c.d.\pi\leq2$ then $k_1(P)=k_1(Q)=0$.
\end{proof}

When there is only one boundary $S^2$ 
then $\mathrm{Ker}(\alpha_P)$ is cyclic,
and the generator is well-defined up to multiplication by a unit.
If $\pi$ is free such units lie in $\pm\pi$,
corresponding to choices of orientation and path to a base-point.
There is then no difficulty in finding $g$. 
(This follows also from the ``uniqueness of top cells" argument of 
\cite[Corollary 2.4.1]{Wa}.)

If $\partial{P}\not=\emptyset$ and $Q$ is aspherical then we may argue instead 
that $c.d.\pi\leq2$ and that $\theta$ may be realized by a map $f:P\to{Q}$.
Since $C_*(P;\mathbb{Z}[\pi])$ is chain homotopy equivalent to 
a finite projective complex of length 2 and $c.d.\pi\leq2$,
the $\mathbb{Z}[\pi]$-module $\pi_2(P)$ is finitely generated and projective.
Conditions (1) and (2) imply that $\chi(P)=\chi(Q)$, 
and so $\mathbb{Z}\otimes_\pi\pi_2(P)=0$.
Hence $\pi_2(P)=0$, since $c.d.\pi\leq2$, and so $f$ is a homotopy equivalence.

Let $(D_k,\partial{D_k})=\sharp^k(D^3,S^2)$ be the 3-sphere with $k$ holes.

\begin{cor}
\label{punct}
Let $(P,\partial{P})$ be a $PD_3$-pair with no $RP^2$ boundary components and such that either $\pi=\pi_1(P)$ is infinite and virtually free or $c.d.\pi=2$.
Then $(P,\partial{P})\simeq
(\widehat{P},\partial\widehat{P})\sharp(D_{m(P)},\partial{D_{m(P)}})$.
\end{cor}

\begin{proof}
We shall compare the $k$-invariants of $P$ and 
$P'=\widehat{P}\sharp{D_m}$ via the space $P_o$ obtained by deleting 
a small open disc from the interior of a collar neighbourhood of $\partial{P}$. 
Clearly $P=P_o\cup{D^3}$, while $P'=P_o\cup{mD^3}\cup{D_{m+1}}$.
Let $\theta:\pi_1(P)\to\pi_1(P')$ be the isomorphism determined by 
the inclusions of $P_o$ into each of $P$ and $P'$, 
and let $g:\pi_2(P)\cong\pi_2(P')$ be the isomorphism constructed 
as in the theorem.

We see from the exact sequence of homology for the pair $(P,P_o)$ 
with coefficients $\mathbb{Z}[\pi]$ that $\pi_2(P_o)=H_2(P_o;\mathbb{Z})$
is an extension of $\pi_2(P)$ by $H_3(D,S^2;\mathbb{Z}[\pi])\cong\mathbb{Z}[\pi]$.
Similarly,
there is an exact sequence 
\[
0\to\mathbb{Z}[\pi]^m\to\pi_2(P_o)\to\pi_2(P')\to\mathbb{Z}[\pi]^{m-1}\to0.
\]
If $\pi$ is virtually free then $H^i(\pi;\mathbb{Z}[\pi]^n)=0$ for all $i\geq2$
and all $n\geq0$, and so the inclusions of $P_o$ into each of $P$ and $P'$
induce isomorphisms $H^3(\pi;\pi_2(P_o))\cong{H^3(\pi;\pi_2(P))}$ and
$H^3(\pi;\pi_2(P_o))\cong{H^3(\pi;\pi_2(P'))}$. 
The $k$-invariants restrict to $k_1(P_o)$, and 
$\theta^*k_1(P')=g_\#k_1(P)$, up to automorphisms.
Hence $(P',\partial{P'})\simeq(P,\partial{P})$, by the theorem.

Essentially the same argument applies if  $c.d.\pi\leq2$, 
but in this it is simpler to observe that 
the $k$-invariants are trivial since $H^3(\pi;\Pi)=0$,
and the inclusion of $P$ into $\widehat{P}$ induces an isomorphism
$\pi\cong\pi_1(\widehat{P}\sharp{D_{m(P)}})$ which respects the 
nontrivial peripheral data.
\end{proof}

With Theorem \ref{asph boundary}  this leads to 
a complement to Theorem \ref{vf wsplit}.

\begin{cor}
\label{vf notsplit cor}
Let $(P,\partial{P})$ be a $PD_3$-pair such that $\pi=\pi_1(P)$
 is virtually free and $w=w_1(P)$ does not split.
Then $(P,\partial{P})$ is a connected sum
${Q\sharp(R,\partial{R}})\sharp(D_m,\partial{D_m})$,
where $Q$ is a $PD_3$-complex and $R$ is a connected sum of copies of
$(D^2\times{S^1},T)$ or $(D^2\tilde\times{S^1}, Kb)$.
\end{cor}

\begin{proof}
It is enough to show that $(\widehat{P},\partial\widehat{P})\simeq
{Q\sharp(R,\partial{R}})$ for some such $Q$ and $(R,\partial{R})$,
by Corollary \ref{punct}.
Since  $(\widehat{P},\partial\widehat{P})$ is peripherally torsion free
and $\pi$ is virtually free, 
the image $F_j$ of $\kappa_j$ is free of rank 
$r_j=\frac12\beta_1(\partial\widehat{P}_j)$, 
for all $j$, and so determines a geometric system $(F_j,\kappa_j')$,
where $\kappa_j':S_j\to{F_j}$ is the composite of $\kappa_j$ 
with projection onto a free factor of $\pi$.
These geometric systems may be realized by pairs $R_j$
which are boundary connected sums 
of pairs with fundamental group $\mathbb{Z}$.
Let $(R,\partial{R})$ be the interior connected sum of such pairs.

The subgroups $F_j$ are independent free factors of $\pi$, 
by Theorem \ref{asph boundary}, 
and so $\pi\cong\sigma*(*_j(F(r_j))$, where $\sigma$ is virtually free. 
Moreover, $\sigma$ is a free factor of $\pi_1(DP)$, 
where $DP$ is the double of $P$ along its boundary,
and so $\sigma\cong\pi_1(Q)$, for some $PD_3$-complex $Q$, by Tura'ev's Splitting Theorem. 
We may assume that $w_1(Q)=w|_\sigma$.
Putting all of the ingredients together gives our result, by Theorem
\ref{S^2 components}.
\end{proof}

See \cite[Chapters 6 and 7]{Hi} for the possibilities for $\pi_1(Q)$ in this case.
If $\pi$ is free then $(P,\partial{P})$ is homotopy equivalent to a 3-manifold pair.
In the light of Lemma \ref{vfree wnont}, 
``$w$ does not split" is probably needed only to exclude summands 
of the form $RP^2\times([0,1],\{0,1\})$ or $RP^2\times{S^1}$.

In general, we must expect that the $k$-invariant may be non-trivial,
and it is not clear how it should relate to the homological invariant $\mu$.
If $\pi$ is torsion free and $c.d.\pi>2$ then $c.d.\pi=3$, 
and $\pi_2(P)$ and $\pi_2(P')$ are projective, 
by the argument of the theorem.
However, $H^3(\pi;\mathbb{Z}[\pi])$ is non-trivial,
and we cannot conclude that the $k$-invariants correspond.

Theorem \ref{asph boundary} extends easily to peripheral systems corresponding to pairs for which $w$ does not split.
Since these have no $RP^2$ boundary components,
it suffices to first realize a pair with aspherical boundary components and then take the connected sum with $(D_m,\partial{D_m})$.

\section{$RP^2$ boundary components}

The strategy of Theorem \ref{S^2 components} should apply also 
when there are boundary components which are copies of $RP^2$, 
but we have not yet been able to identify the extension relating 
$\pi_2(P)$ to the peripheral data via duality.
If  $X$ is a cell complex and $f:RP^2\to{X}$ then $H^2(RP^2;f^*\pi_2(X))$ acts 
simply transitively on the set $[RP^2,X]_\theta$ of based homotopy
classes of based maps such that $\pi_1(f)=\theta$ \cite{Wh49}.
(Note that self-maps of $RP^2$ which induce the identity on $\pi_1$ 
lift to self maps of $S^2$ of {\it odd\/} degree,
and so the map $h\mapsto{h_*[RP^2]}$ from $[RP^2,RP^2]_{id}$
to $H_2(RP^2;\mathbb{Z}^w)$ is injective, but not onto.)
The corresponding summands of $H_2(\partial{P};\mathbb{Z}[\pi])$ have
the form $L_w=\mathbb{Z}[\pi]/\mathbb{Z}[\pi](w+1)$,
where $w$ is the image of the generator of the $RP^2$ in question,
and are not free $\mathbb{Z}[\pi]$-modules.
It is not clear how to determine the extension of $A$ by $\pi_2(\partial{P})$
giving $\Pi$.

Let $v\in\pi$ be such that $v^2=-1$ and $w(v)=-1$.
Then $\pi\cong\pi^+\rtimes\langle{v}\rangle$.
Let $\Gamma=\mathbb{Z}[\pi]$ and $\Gamma^\pm=\Gamma.(v\pm1)$.
Then $\Gamma^\pm\cong\Gamma/\Gamma^\mp$, and 
$f(\gamma)=(\gamma(v+1),\gamma(v-1))$ and
$g(\gamma(v+1),\delta(v-1))=\gamma(v+1)-\delta(v-1)$
define homomorphisms $f:\Gamma\to\Gamma^+\oplus\Gamma^-$
and $g:\Gamma^+\oplus\Gamma^-\to\Gamma$ such that $fg$ and $gf$ 
are multiplication by 2.

Using this near splitting of the group ring,
we can show that if $1<v.c.d.\pi<\infty$ then $Ext(A,L_w)$ has exponent 2,
while if $\pi$ is virtually free then $Ext(A,L_w)\cong\mathbb{Z}$
up to torsion of exponent 2.
We do not yet have a clear result.

\section{finite fundamental group}

A $PD_3$-complex $X$ with finite fundamental group is orientable,
and is determined by $\pi=\pi_1(X)$ and $k_2(X)\in{H^4(\pi;\mathbb{Z})}$,
which is now the first non-trivial $k$-invariant, since $\pi_2(X)=0$.
Here $\pi$ may be any finite group with cohomological period dividing 4,
and $k_2(P)$ may be any generator of $H^4(\pi;\mathbb{Z})\cong\mathbb{Z}/|G|\mathbb{Z}$ \cite{Wa}.
We shall show that every orientable $PD_3$-pair with finite 
fundamental group may be obtained by puncturing the top cell of 
a $PD_3$-complex with the same group.

\begin{lemma}
Let $(P,\partial{P})$ be a $PD_3$-pair such that $\pi=\pi_1(P)$ is finite.
Then the components of $\partial{P}$ are copies of $S^2$ or $RP^2$.
\end{lemma}

\begin{proof}
Since $\pi$ is finite,
$H_1(P;\mathbb{Z})$ and $H^1(P;\mathbb{Z}^w)$ are both finite,
and so $H_1(\partial{P};\mathbb{Z})$ is also finite.
\end{proof}

\begin{theorem}
Let $(P,\partial{P})$ be an orientable $PD_3$-pair with $\pi=\pi_1(P)$ finite.
Then $(P,\partial{P})\simeq\widehat{P}\sharp(D_{m(P)},\partial{D_{m(P)}})$,
where $\widehat{P}$ is a $PD_3$-complex.
\end{theorem}

\begin{proof}
Since $\pi$ is finite and $(P,\partial{P})$ is orientable, 
$\partial{P}=m(P)S^2$, and so $\widehat{P}=P\cup{m(P)D^3}$
is a $PD_3$-complex.
We may assume that $\partial{P}$ is non-empty, 
and so $\pi_2(P)\cong\Pi=\mathbb{Z}[\pi]^{m(P)}/\Delta(\mathbb{Z})$,
where $\Delta:\mathbb{Z}\to\mathbb{Z}[\pi]^{m(P)}$ is the ``diagonal" monomorphism.

If $\pi=1$ then $(P,\partial{P})\simeq(D_{m(P)},\partial{D_{m(P)}})$ \cite[\S3.5]{Hi}.
In general, $P$ is determined by $\pi$, $m(P)$ and $k_1(P)$.
These data also determine the inclusion of the boundary, 
and hence the homotopy type of the pair.

The $k$-invariant $k_1(P)$ is the extension class of the sequence
\[
0\to\mathbb{Z}[\pi]^{m(P)}/\Delta(\mathbb{Z})
\to{C_2}\to{C_1}\to\mathbb{Z}[\pi]\to\mathbb{Z}\to0
\]
in $H^3(\pi;\Pi)=Ext_{\mathbb{Z}[\pi]}^3(\mathbb{Z},\Pi)$.
The connecting homomorphism in the long exact sequence
associated to the coefficient sequence
\[
0\to\mathbb{Z}\to\mathbb{Z}[\pi]^{m(P)}\to\Pi\to0
\]
gives an isomorphism
$H^3(\pi;\Pi)\cong{H^4(\pi;\mathbb{Z})}=Ext_{\mathbb{Z}[\pi]}^4(\mathbb{Z},\mathbb{Z})$
which sends $k_1(P)$ to the extension class of the sequence
\[
0\to\mathbb{Z}\to\mathbb{Z}[\pi]^{m(P)}\to{C_2}\to{C_1}\to\mathbb{Z}[\pi]\to\mathbb{Z}\to0.
\]
This is $k_2(\widehat{P})$,
and so $k_1(P)=k_1(\widehat{P}\sharp(D_{m(P)})$
(up to the actions of $Aut(\pi)$ and $Aut_\pi(\Pi))$).
Hence $(P,\partial{P})\simeq(\widehat{P}\sharp(D_{m(P)},\partial{D_{m(P)}})$,
since they have isomorphic fundamental groups, the same number of boundary components
and equivalent first $k$-invariants.
\end{proof}

It is easy to show that a finite group $G$ with cohomological period
dividing 4 and having a {\it free\/} periodic resolution satisfies Turaev's condition 
for some choice of homology class $\mu$ \cite{Tu}.
Let 
\begin{equation*}
\begin{CD}
0\to\mathbb{Z}\to{C_3}\to{C_2}@>\partial_2>>{C_1}\to{C_0}\to\mathbb{Z}\to0
\end{CD}
\end{equation*}
be an exact sequence of $\mathbb{Z}[G]$-modules in which the $C_i$ are 
finitely generated free modules.
Then the $\mathbb{Z}$-linear dual of this sequence is also exact.
Composition with the additive function $c:\mathbb{Z}[G]\to\mathbb{Z}$
given by $c(\Sigma{n_g}g)=n_1$ defines natural isomorphisms
$M^\dagger=\overline{Hom_{\mathbb{Z}[G]}(M,\mathbb{Z}[G])}\cong
{Hom_\mathbb{Z}(M,\mathbb{Z})}$,
and so the $\mathbb{Z}$-linear dual of the complex $C_*$ is also the 
$\mathbb{Z}[G]$-linear dual of $C_*$.
A Schanuel's Lemma argument then shows that
$\mathrm{Cok}(\partial_2)$ and $\mathrm{Cok}(\partial_2^\dagger)$
are stably isomorphic.
However the standard construction of a $PD_3$-complex realizing $G$ 
(as in \cite{Wa}) is more direct than one involving an appeal to that theorem.

There is only one non-orientable example (up to punctures).

\begin{theorem}
\label{nonor finite}
Let $(P,\partial{P})$ be a $PD_3$-pair with $\pi=\pi_1(P)$ finite and 
which is not orientable.
Then
$\pi\cong\mathbb{Z}/2\mathbb{Z}$ and 

$(P,\partial{P})\simeq
(RP^2\times([0,1],\{0,1\}))\sharp(D_{m(P)},\partial{D_{m(P)}})$.
\end{theorem}

\begin{proof}
Since $\pi$ is finite, the boundary components must be either $S^2$ or $RP^2$.
Suppose first that $m(P)=0$.
Then $\partial{P}=rRP^2$ for some $r$, 
which must be even since $\chi(\partial{P})=2\chi(P)$
and strictly positive since $P$ is non-orientable.
The inclusion $\iota$ of a boundary component splits the orientation character,
and so $\pi\cong\pi^+\rtimes\mathbb{Z}/2\mathbb{Z}$.
Let $Q$ be the (irregular) covering space with fundamental group 
$\mathrm{Im}(\pi_1(\iota))$, 
and let $\partial{Q}$ be the preimage of $\partial{P}$ in $Q$.
Then $\partial{Q}=r|\pi^+|RP^2$ and $(Q,\partial{Q})$ is a $PD_3$-pair.
Let $DQ=Q\cup_{\partial{Q}}Q$ be the double of $Q$ along its boundary.
Then $DQ$ is a non-orientable $PD_3$-complex and 
$\pi_1(DQ)\cong{F(s)}\times\mathbb{Z}/2\mathbb{Z}$,
where $s=r|\pi^+|-1$.
Since $r\geq2$ and $s\leq1$, by the Centralizer Theorem of Crisp \cite{Cr00},
we must have $r=2$ and $\pi^+=1$.

We now allow $\partial{P}$ to have $S^2$ components.
On applying the paragraph above to the pair obtained by capping these off, 
we see that $\partial{P}$ has two $RP^2$ components and 
$\pi\cong\mathbb{Z}/2\mathbb{Z}$.
Hence $P^+\simeq\vee^{2m(P)+1}S^2$ and 
$\pi_2(P)\cong\mathbb{Z}[\pi]^{m(P)}\oplus\mathbb{Z}^-$.
The inclusion of the $S^2$ boundary components and one of the two $RP^2$
boundary components induces a homomotopy equivalence
$(\vee^{m(P)}S^2)\vee{RP^2}\simeq{P}$.
The inclusion of the other boundary component then corresponds to a class in 
$H^2(RP^2;\pi_2(P))\cong\mathbb{Z}^{m(P)+1}$.

Let $h:RP^2\to{P}$ be the inclusion of the other boundary component.
Then $\theta=\pi_1(h)$ is an isomorphism,
and composition of Poincar\'e duality for $RP^2$ with the Hurewicz homomorphism 
for $P$ gives an isomorphism $\rho:H^2(RP^2;\theta^*\pi_2(P))\to{H_2}(P;\mathbb{Z}^w)$.
The group $H^2(RP^2;\theta^*\pi_2(P))\cong\mathbb{Z}^{m(P)+1}$ 
acts simply transitively on $[RP^2,P]_\theta$.
If $b:RP^2\to{P}$ is a $\pi_1$-injective map and $x.b$ is the map 
obtained by the action of $x\in{H^2(RP^2;\theta^*\pi_2(P))}$ then
$(x.b)_*[RP^2]=b_*[RP^2]+2\rho(x)$.
Hence $b\mapsto{b_*[RP^2]}$ is injective,
and so $h$ is determined by the boundary compatibility condition,
that $h_*[RP^2]$ be the negative of the sum of the images of 
the fundamental classes of the other boundary components.
\end{proof}

\section{when does the group determine the pair?}

If $G$ is a finitely generated group then there are at most finitely many 
homeomorphism types of bounded, compact, irreducible orientable 3-manifolds $M$ 
such that $\pi_1(M)\cong{G}$, and there only finitely many pairs 
$(G,\{\kappa_j\})$ 
which are peripheral systems of 3-manifold pairs.
These results are consequences of the Johannson Deformation Theorem
\cite{Sw80}.
Are there analogues for $PD_3$-pairs?
If $(P,\partial{P})$ is a $PD_3$-pair are there only finitely many homotopy types 
of $PD_3$-pairs $(Q,\partial{Q})$ with aspherical boundary
such that $\pi_1(Q)\cong\pi_1(P)$?

This appears to be not known even when $P$ is aspherical and the boundary
is $\pi_1$-injective, although an analogue of Johannson's Deformation Theorem
for $PD_3$-group pairs has recently been proven \cite{RSS20}.
We can answer this question in one rather special case.
The peripheral system $(\pi,\{\kappa_j\})$ of a $PD_3$-pair is {\it atoroidal\/}
if every $\mathbb{Z}^2$ subgroup of $\pi$ is conjugate into
$\mathrm{Im}(\kappa_j)$, for some $j$, 
and $\pi$ is not virtually $\mathbb{Z}^2$.

\begin{theorem}
Let $(P,\partial{P})$ be an orientable $PD_3$-pair such that $P$ is aspherical,
$\pi=\pi_1(P)$ has one end and $\chi(P)=0$.
Assume that the peripheral system of the pair is atoroidal.
Then any $PD_3$-pair $(Q,\partial{Q})$ with aspherical boundary and
such that $\pi_1(Q)\cong\pi_1(P)$ is homotopy equivalent to $(P,\partial{P})$.
\end{theorem}

\begin{proof}
Since $P$ is aspherical and orientable, $\pi$ has one end and $\chi(P)=0$,
the components of $\partial{P}$ are tori.
Since $\pi$ has one end and the boundary of $Q$ is aspherical,
$Q$ is aspherical, and its peripheral system is $\pi_1$-injective. 
(These assertions follow immediately on considering the exact sequence
relating $\Pi$ and $\overline{E)(\pi)}$ in \S3 above.)
In particular, $Q\simeq{P}$, and so $\chi(Q)=\chi(P)=0$.
Since $\chi(\partial{Q})=2\chi(Q)=0$,
the components of $\partial{Q}$ are tori or Klein bottles.
Every $\mathbb{Z}^2$ subgroup of $\pi$ is conjugate into 
the image of a boundary component of $P$,
since the peripheral system of $(P,\partial{P})$ is atoroidal.

The pair $(P,\partial{P})$ is not of $I$-bundle type,  
since $\chi(\pi)=0$ and $\pi$ is not virtually $\mathbb{Z}^2$,
and so the images of the boundary components are 
their own commensurators in $\pi$,
and are pairwise non-conjugate \cite[Lemma 2.2]{KR88}.
Hence $Q$ is not of $I$-bundle type either,
its boundary has no Klein bottle components,
and the images of the boundary components of $P$ and of $Q$ are 
maximal free abelian subgroups of $\pi$.
Hence the images of the components of $\partial{Q}$ in $\pi$ are conjugate to 
images of boundary components of $P$.
The images of the fundamental classes of the boundary components of $P$
in $H_2(\pi;\mathbb{F}_2)$ have sum 0 and generate a subspace of dimension 
$\beta_0(\partial{P})-1$,  
by the boundary compatibility condition.
Since the same holds for the boundary components of $Q$, 
the boundaries must correspond bijectively,
and so a homotopy equivalence $P\simeq{Q}$ induces a homotopy equivalence of pairs.
\end{proof}

The exteriors of the reef knot $3_1\#-3_1$
and the granny knot $3_1\#3_1$ provide a well-known example
of a pair of aspherical orientable 3-manifolds with isomorphic fundamental groups,  but which are not homotopy equivalent {\it rel\/} boundary.
There are simpler examples if we allow non-orientable pairs.
If $F$ is an aspherical closed orientable surface and $M$ is 
the total space of a non-trivial $I$-bundle over $F$ then $M\simeq{F}$,
but $F\times(I,\partial{I})$ and $(M,\partial{M})$ are not homotopy
equivalent as pairs.
(In fact $\partial{M}$ is connected and $(M,\partial{M})$ is non-orientable.)

\end{document}